\documentclass[11pt,a4paper,oneside]{amsart}

\usepackage{amsaddr}
\usepackage{amssymb}
\usepackage{bm}
\usepackage{booktabs}
\usepackage{commath}
\usepackage{colortbl}
\usepackage{dsfont}
\usepackage{enumerate}
\usepackage{esint}
\usepackage[left=0.09\paperwidth, right=0.15\paperwidth, bottom=0.12\paperheight]{geometry}
\usepackage{graphicx}
\usepackage[displaymath, mathlines]{lineno}
\usepackage{mathtools}
\usepackage{mathrsfs}
\usepackage{multirow,bigdelim}
\usepackage{url}
\usepackage{pdflscape}
\usepackage{rotating}
\usepackage[color=white,linecolor=black, textwidth=0.12\paperwidth]{todonotes}
\usepackage{wrapfig}

\numberwithin{equation}{section}
\newtheorem{theorem}{Theorem}
\numberwithin{theorem}{section}
\newtheorem{lemma}[theorem]{Lemma}

\newtheorem{prop}[theorem]{Proposition}
\theoremstyle{definition}
\newtheorem*{remarkx}{Remark}

\newenvironment{remark}
  {\pushQED{\qed}\remarkx}
  {\popQED\endremarkx}

\usepackage{color}
\newcommand{\R}{\mathbb{R}}
\newcommand{\N}{\mathbb{N}}

\newcommand{\Q}{\mathbb{Q}}
\newcommand{\E}{\mathcal{E}}

\renewcommand{\P}{\mathbb P}

\DeclareMathOperator{\osc}{osc}

\DeclareMathOperator{\vol}{vol}

\DeclareMathOperator{\Div}{div}
\renewcommand{\div}{\Div}
\renewcommand{\Delta}{\triangle}

\usepackage{stmaryrd}
\newcommand{\jump}[1]{\left\llbracket #1 \right\rrbracket}

\newcommand{\T}{\mathcal{T}}
\newcommand{\V}{\mathcal{V}}

\definecolor{airforceblue}{rgb}{0.36, 0.54, 0.66}
\definecolor{ao(english)}{rgb}{0.0, 0.5, 0.0}

\DeclareFontFamily{U}{matha}{\hyphenchar\font45}
\DeclareFontShape{U}{matha}{m}{n}{
      <5> <6> <7> <8> <9> <10> gen * matha
      <10.95> matha10 <12> <14.4> <17.28> <20.74> <24.88> matha12
      }{}
\DeclareSymbolFont{matha}{U}{matha}{m}{n}
\DeclareFontSubstitution{U}{matha}{m}{n}

\DeclareFontFamily{U}{mathx}{\hyphenchar\font45}
\DeclareFontShape{U}{mathx}{m}{n}{
      <5> <6> <7> <8> <9> <10>
      <10.95> <12> <14.4> <17.28> <20.74> <24.88>
      mathx10
      }{}
\DeclareSymbolFont{mathx}{U}{mathx}{m}{n}
\DeclareFontSubstitution{U}{mathx}{m}{n}

\DeclareMathDelimiter{\vvvert}{0}{matha}{"7E}{mathx}{"17}

\DeclareMathOperator{\ext}{ext}
\DeclareMathOperator{\eff}{eff}
\DeclareMathOperator{\loc}{loc}
\DeclareMathOperator{\irr}{int}
\DeclareMathOperator{\RT}{RT}
\DeclareMathOperator{\meas}{meas}
\newcommand{\bbone}{\mathds{1}}
\newcommand{\seminorm}[1]{\vvvert #1 \vvvert}
\newcommand{\ltwonorm}[1]{\lVert #1 \rVert}
\newcommand{\hta}{{\check \T_a}}
\newcommand{\hT}{{\check T}}
\newcommand{\he}{{\check e}}
\newcommand{\hE}{{\check \E}}
\newcommand{\dr}{{\tilde r}}
\renewcommand{\emptyset}{{\varnothing}}
\newcommand{\hprime}[1]{H^1_{*, #1}(\hta)}
\newcommand{\laplace}{\mathcal{4}}

\renewcommand{\vec}{\bm}
\renewcommand{\norm}[1]{\lVert #1 \rVert}
\newcommand{\refsq}{{\hat T}}
\renewcommand{\sp}[1]{{\mathscr #1}} 
\newcommand{\mat}[1]{{\bm #1}} 
\newcommand{\ltwoinp}[4][\vec x]{\langle #2, #3\rangle _{#4}}
\newcommand{\honeinp}[3]{\langle \nabla #1, \nabla #2\rangle _{#3}} 
\newcommand{\honeinpx}[3]{\langle \tfrac{\dif}{\dif x} #1, \tfrac{\dif}{\dif x} #2\rangle _{#3}} 

\DeclareFontFamily{U}{mathx}{\hyphenchar\font45}
\DeclareFontShape{U}{mathx}{m}{n}{
      <5> <6> <7> <8> <9> <10>
      <10.95> <12> <14.4> <17.28> <20.74> <24.88>
      mathx10
      }{}
\DeclareSymbolFont{mathx}{U}{mathx}{m}{n}
\DeclareFontSubstitution{U}{mathx}{m}{n}
\DeclareMathAccent{\widecheck}{0}{mathx}{"71}
\DeclareMathAccent{\wideparen}{0}{mathx}{"75}

\title{On $p$-robust saturation on quadrangulations}
\author{Jan Westerdiep}
\email{j.h.westerdiep@uva.nl}
\address{Korteweg-de Vries Institute for Mathematics, University of Amsterdam, \\
P.O. Box 94248, 1090 GE Amsterdam, The Netherlands}
\date{\today}
\subjclass[2010]{}

\thanks{\emph{Funding:} The author was supported by the Netherlands Organisation for Scientific Research (NWO) under contract.~no.~613.001.652}

\begin{document}
\begin{abstract}
For the Poisson problem in two dimensions, posed on a domain partitioned into
axis-aligned rectangles with up to one hanging node per edge, we envision an
efficient error reduction step in an instance-optimal $hp$-adaptive finite
element method. Central to this is the problem: Which increase in local polynomial degree
ensures $p$-robust contraction of the error in energy norm? We reduce this problem to a small
number of saturation problems on the reference square, and provide strong numerical
evidence for their solution.

  ~~

  \smallskip
  \noindent\textsc{\subjclassname.}
    65N12, 
    65N30, 
    65N50. 

  \smallskip
  \noindent\textsc{\keywordsname.}
    Poisson equation, $p$-robustness, $hp$-adaptive finite element method, convergence, 1-irregular quadrilateral meshes.
\end{abstract}

\maketitle

\section{Introduction}
We consider the Poisson model problem of finding $u: \Omega \to \R$ that satisfies
\begin{equation}
  - \laplace u = f ~~ \text{in $\Omega$}, \quad u = 0 ~~ \text{on $\partial \Omega$},
  \label{eqn:poisson}
\end{equation}
where $\Omega \subset \R^2$ is a connected union of a finite number of essentially disjoint axis-aligned rectangles, and $f \in L_2(\Omega)$. Given a 1-irregular quadrangulation $\T$ of the domain into essentially disjoint axis-aligned rectangles, let $U_\T$ be the space of continuous piecewise polynomials of variable degree w.r.t.~$\T$ that vanish on the domain boundary, and let $u_\T \in U_\T$ be its best approximation of $u$ in energy norm.
We are interested in the following \emph{contraction problem}:
\begin{quotation}
  Which ($hp$-)refinement $\overline \T$ of $\T$ ensures \emph{contraction} of the energy error, in that
  \[
    \norm{\nabla u - \nabla u_{\overline \T}}_{L_2(\Omega)} \leq \alpha \norm{\nabla u - \nabla u_\T}_{L_2(\Omega)}
  \]
  for some fixed $\alpha < 1$ independent of $\T$ and its local polynomial degrees?
\end{quotation}
It is well known that this problem is equivalent to the \emph{saturation problem} of finding $\overline \T$ for which $\norm{\nabla u_{\T} - \nabla u_{\overline \T}}_{L_2(\Omega)} \leq \rho \norm{\nabla u - \nabla u_{\T}}_{L_2(\Omega)}$ for some $\rho > 1$; in this work, we will study the saturation problem, posed locally on a patch of rectangles around a given vertex.

The idea of $hp$-adaptive finite element methods started gaining momentum in the eighties with the seminal works of Babu\v{s}ka and colleagues \cite{Guo1986,Guo1986a}, where they showed that for certain elliptic boundary value problems, careful a priori decisions between $h$-refinement and $p$-enrichment can yield a sequence of finite element solutions that exhibit an exponential convergence rate with respect to the number of degrees of freedom (DoFs).

Since then, a lot of research has been done on $hp$-adaptive refinement driven by a posteriori error estimates, but despite the interest, it was not until 2015 that Canuto \emph{et al.}~\cite{Canuto2016} proved the instance optimality---and with it, exponential convergence---of one such method. The method alternates between (i) a module that refines the triangulation to reduce the energy error with a sufficiently large fixed factor, and (ii) an $hp$-coarsening strategy developed by Binev~\cite{Binev2015} that essentially removes near-redundant degrees of freedom to yield an \emph{instance optimal} triangulation. The sequence of triangulations found after each $hp$-coarsening step then exhibits the desired exponential decay.

In \cite{Canuto2016}, the error reducer of step (i) was a typical $h$-adaptive loop driven by an element-based D\"orfler marking, using the a posteriori error estimator of Melenk and Wohlmuth \cite{Melenk2001}. The efficiency of this error estimator is known to be sensitive to polynomial degrees, which can lead to a runtime that grows exponentially in the number of DoFs.

In \cite{Canuto2017}, Canuto \emph{et al.}~explore a different error reduction strategy. It is an adaptive $p$-enrichment loop driven by a vertex-based D\"orfler marking using the equilibrated flux estimator, which was shown to be $p$-robust in \cite{Braess2009b}. They show that solving a number of \emph{local saturation problems}, posed on patches around a vertex in terms of dual norms of residuals, leads to an efficient error reducer.
They were able to reduce the problem, stated on triangulations without hanging nodes, to three problems on a reference triangle, and provided numerical results indicating that uniform saturation holds when increasing the local degree $p$ to $p + \lceil \lambda p \rceil$ for any constant $\lambda > 0$, but that an additive quantity of the form $p + n$ is insufficient.

Finally, in~\cite{Canuto2017a}, Canuto \emph{et al.}~present a theoretical result solving slightly ill-fitted variant on one of the reference problems. Whereas the former two works discuss partitions of the domain into triangles, the latter proves a result on the reference \emph{square} instead. As a first step towards repairing this inconsistency, the present work considers quadrangulations. Our goal of adaptive approximation requires us to consider partitions with hanging nodes, which introduce complications. A key contribution in this regard has been made by Dolej\v{s}\'i \emph{et al.}~in \cite{Dolejsi2016}.

\subsection*{Contributions of this work}
This work has two related goals. In a larger context, we aim to take a step in the direction of a polynomial-time $hp$-adaptive FEM with exponential convergence rates. In particular, we are interested in finding an efficient error reducer. To this end, we reduce the \emph{saturation problem} to a small number of problems on the reference square, and provide numerical results suggesting these problems may be solvable theoretically. We detail on the computational aspect as well, so that the numerical results are easily reproducible.

On a lower level, this work aims to extend the reduction to reference problems of \cite{Canuto2017} from regular triangulations equipped with polynomials of certain \emph{total degree} to the situation of 1-irregular quadrangulations with polynomials of certain \emph{degree in each variable separately}. Allowing 1-irregularity makes for a rather involved adaptation of the original result, as the refined regular patches are not necessarily composed of elements containing the original vertex.

\subsection*{Organisation of this work}
In \S 2, we will establish our notation. In \S 3, we show a contractive property within ($hp$-)adaptive finite element context, under a local patch-based saturation assumption. In \S 4, we reduce the local saturation assumption to boundedness of a small number of reference saturation coefficients. In \S 5, we discuss the computation of these coefficients, and in \S 6 we show numerical results suggesting that these quantities are in fact bounded.

\section{Notation and setup}
In this work, $A \lesssim B$ will mean that $A$ may be bounded by a multiple of $B$, independently of the parameters of $A$ and $B$, and $A \eqsim B$ means that $A \lesssim B$ and $B \lesssim A$.

\subsection{Quadrangulations}
We consider partitions $\T$ of the domain into closed axis-aligned rectangles. We impose that $T_1^{\circ} \cap T_2^{\circ} = \emptyset$ for $T_1, T_2 \in \T$ distinct, and allow \emph{irregularity} along shared edges, meaning that $T_1 \cap T_2$ may be empty, a shared vertex, or part of a shared edge. Irregularity allows for highly adaptive quadrangulations, but to ensure $p$-robustness of our main result, we restrict ourselves to \emph{1-irregularity}: every element edge may contain up to one \emph{hanging node}---a vertex in the interior of a neighbours edge.

To avoid pathological situations, we lastly assume that every $\T$ is found from a regular initial quadrangulation (i.e., without hanging nodes) by means of repeated red-refinement (subdivision into four similar rectangles), thus automatically ensuring uniform shape regularity. We collect the family of such quadrangulations in the set $\mathbb T$. See Figure~\ref{fig:allquadrangulations} for a few examples.

The set of \emph{nonhanging} vertices of a quadrangulation $\T \in \mathbb T$ form the set $\V_\T$, and $\V_{\T}^{\ext}$ (resp.~$\V_{\T}^{\irr}$) is its subset of boundary (resp.~interior) vertices. The edges of $\T$ form the set $\E_{\T}$.

\begin{figure}
  \begin{tikzpicture}[scale=1.5]
    \draw (0,0) rectangle (1,1);
    \draw (1,0) rectangle (2,1);
    \draw (0,1) rectangle (1,2);
    \draw (1,1) rectangle (2,2);
    \draw[line width=1mm] (0,0) -- (2,0) -- (2,2) -- (0,2) -- (0,0) -- (2,0);
    \node[below] at (1,0){(a)};
  \end{tikzpicture}\quad
  \begin{tikzpicture}[scale=1.5]
    \draw (0,.5) rectangle (.5,1);
    \draw (.5,0) rectangle (1,.5);
    \draw (.5,.5) rectangle (1,1);

    \draw (0,0) rectangle (.25,.25);
    \draw (0,.25) rectangle (.25,.5);
    \draw (.25,0) rectangle (0.5,0.25);
    \draw (.25,.25) rectangle (0.5,0.5);

    \draw (1,0) rectangle (2,1);
    \draw (0,1) rectangle (1,2);
    \draw (1,1) rectangle (2,2);
    \draw[line width=1mm] (0,0) -- (2,0) -- (2,2) -- (0,2) -- (0,0) -- (2,0);
    \node[below] at (1,0){(b)};
  \end{tikzpicture}\quad
  \begin{tikzpicture}[scale=1.5]
    \draw (0,0) rectangle (.5,.5);
    \draw (.5,0) rectangle (1,.5);
    \draw (.5,.5) rectangle (1,1);

    \draw (0,.5) rectangle (.25,.75);
    \draw (0,.75) rectangle (.25,1);
    \draw (.25,.5) rectangle (.5,.75);
    \draw (.25,.75) rectangle (.5,1);

    \draw (1,0) rectangle (2,1);
    \draw (0,1) rectangle (1,2);
    \draw (1,1) rectangle (2,2);
    \draw[line width=1mm] (0,0) -- (2,0) -- (2,2) -- (0,2) -- (0,0) -- (2,0);
    \node[below] at (1,0){(c)};
  \end{tikzpicture}\quad
  \begin{tikzpicture}
    \draw (0,0) rectangle (1,2);
    \draw (1,0) rectangle (3,1);
    \draw (1,1) rectangle (2,2);
    \draw (2,1) rectangle (3,3);
    \draw (0,2) rectangle (2,3);
    \draw[line width=1mm] (0,0) -- (3,0) -- (3,3) -- (0,3) -- (0,0) -- (3,0);
    \node[below] at (1.5,0){(d)};
  \end{tikzpicture}
  \caption{(a) Regular initial quadrangulation $\T_{a} \in \mathbb T$ of a square domain $\Omega$; (b) 1-irregular quadrangulation found from $\T_a$ through red-refinement; (c) quadrangulation found from $\T_a$ that is not 1-irregular; (d) typical ``pathological'' quadrangulation excluded from this paper.}
  \label{fig:allquadrangulations}
\end{figure}
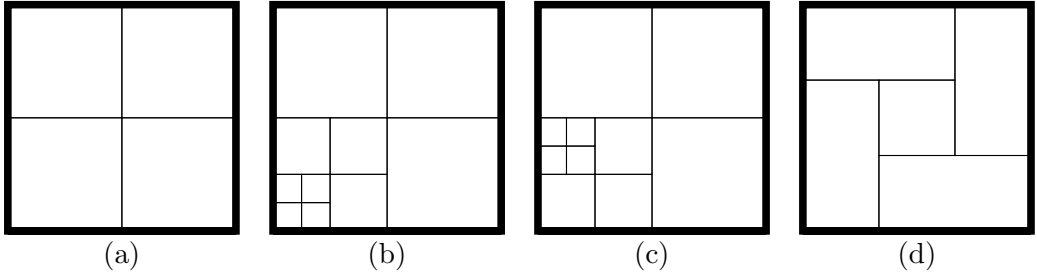

\subsection{Polynomials on quadrangulations}
For $T \in \T \in \mathbb T$, write $\Q_{p,p'}(T)$ for the space of polynomials on $T$ of degree at most $p$ and $p'$ in the two canonical coordinates. Define $\Q_p(T) := \Q_{p,p}(T)$. Equip each $T$ with a local polynomial degree $p_T = p_{T, \T}$ for which $p_T \geq 1$, and write $\vec p_{\T} := (p_T)_{T \in \T}$ for the collection of these local degrees. Then with $\Q_{\vec p_{\T}}^{-1}(\T) := \prod_{T \in \T} \Q_{p_T}(T)$ the space of broken piecewise polynomials over $\T$ of degree at most $p_T$ on every element, we introduce the finite-dimensional subspace $U_\T$ of $H^1_0(\Omega)$ as
\[
  U_\T := H^1_0(\Omega) \cap \Q_{\vec p_\T}^{-1}(\T) \quad (\T \in \mathbb T).
\]
Denote with $u \in H^1_0(\Omega)$ the weak solution to~\eqref{eqn:poisson}, and its Galerkin approximation as $u_\T \in U_\T$.

\subsection{Patches}
Let $\psi_a$ be the \emph{hat function} characterized by $\psi_a \in C(\overline \Omega) \cap \Q_1^{-1}(\T)$ and $\psi_a(b) = \delta_{ab}$ for all $b \in \V_\T$. Let $\omega_a = \omega_{\T,a}$ be its support, and denote with $\T_a \subset \T$ the quadrangulation restricted to $\omega_a$; we call this set a \emph{patch}.
For each nonhanging vertex $a \in \V_{\T}$, write
\[
  \vec p_a := \vec p_{\T_a} = (p_T)_{T \in \T_a}, \quad p_a := \max \vec p_a.
\]
It will prove meaningful to decompose the patch edges $\E_{\T_a} := \set{e \in \E_\T : e \subset \omega_a}$ as
\[
  \E_a^{\ext} := \set{e \in \E_{\T_a}: e \subset \partial \omega_a}, \quad \E_a^{\irr} := \E_{\T_a} \setminus \E_a^{\ext}.
\]
We moreover decompose exterior edges into \emph{Dirichlet} and \emph{Neumann} edges, through
\[
  \E_a^{\ext,D} := \set{e \in \E_a^{\ext}: a \in e}, \quad \E_a^{\ext,N} := \E_a^{\ext} \setminus \E_a^{\ext,D},
\]
giving rise to the local spaces
\[
  H^1_*(\omega_a) := \begin{cases}
    \set{v \in H^1(\omega_a) : \ltwoinp{v}{\bbone}{\omega_a} = 0} & a \in \V_\T^{\irr}, \\
    \set{v \in H^1(\omega_a) : v|_{e} = 0 \text{ on } e \in \E^{\ext,D}_a} & a \in \V_\T^{\ext}.
  \end{cases}
\]
\begin{remark}
  Our definition of $H^1_*(\omega_a)$ differs from its definition in, e.g., \cite{Canuto2017, Dolejsi2016} when $a \in \V^{\ext}_\T$. In previous works, functions in $H^1_*(\omega_a)$ vanish on the entire part $\partial \omega_a \cap \partial \Omega$; in our case, they vanish only on those edges $e \subset \partial \omega_a \cap \partial \Omega$ for which $a \in e$. This altered definition was convenient for our proof, and relevant dual norm properties of the residual in \S\ref{sec:reduce} carry over to our case.
\end{remark}

\subsection{Refined patches}
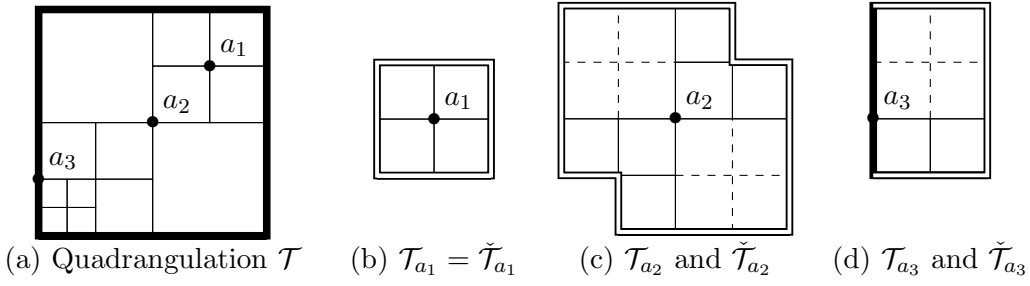
\begin{figure}
  \begin{tikzpicture}[scale=1.5]
    \draw (0,.5) rectangle (.5,1);
    \draw (.5,0) rectangle (1,.5);
    \draw (.5,.5) rectangle (1,1);

    \draw (0,0) rectangle (.25,.25);
    \draw (0,.25) rectangle (.25,.5);
    \draw (.25,0) rectangle (0.5,0.25);
    \draw (.25,.25) rectangle (0.5,0.5);

    \draw (1,0) rectangle (2,1);
    \draw (0,1) rectangle (1,2);

    \draw (1,1) rectangle (1.5,1.5);
    \draw (1,1.5) rectangle (1.5,2);
    \draw (1.5,1) rectangle (2,1.5);
    \draw (1.5,1.5) rectangle (2,2);

    \node at (1,1){$\bullet$};
    \node[above right] at (1,1){$a_2$};

    \node at (1.5,1.5){$\bullet$};
    \node[above right] at (1.5,1.5){$a_1$};

    \node at (0,0.5){$\bullet$};
    \node[above right] at (0,0.5){$a_3$};

    \node[below] at (1,0){(a) Quadrangulation $\T$};
    \draw[line width=1mm] (0,0) -- (2,0) -- (2,2) -- (0,2) -- (0,0) -- (2,0);
  \end{tikzpicture}\quad
  \begin{tikzpicture}[scale=1.5]

    \draw (1,1) rectangle (1.5,1.5);
    \draw (1,1.5) rectangle (1.5,2);
    \draw (1.5,1) rectangle (2,1.5);
    \draw (1.5,1.5) rectangle (2,2);

    \node at (1.5,1.5){$\bullet$};
    \node[above right] at (1.5,1.5){$a_1$};

    \node[below] at (1.5,0.5){(b) $\T_{a_1} = \check \T_{a_1}$};

    \draw[line width=1mm] (1,1) -- (2,1) -- (2,2) -- (1,2) -- (1,1) -- (2,1);
    \draw[line width=0.5mm,white] (1,1) -- (2,1) -- (2,2) -- (1,2) -- (1,1) -- (2,1);  
  \end{tikzpicture}\quad
  \begin{tikzpicture}[scale=1.5]
    \draw (0,.5) rectangle (.5,1);
    \draw (.5,0) rectangle (1,.5);
    \draw (.5,.5) rectangle (1,1);

    \draw (1,0) rectangle (2,1);
    \draw (0,1) rectangle (1,2);

    \draw (1,1) rectangle (1.5,1.5);
    \draw (1,1.5) rectangle (1.5,2);
    \draw (1.5,1) rectangle (2,1.5);

    \node at (1,1){$\bullet$};
    \node[above right] at (1,1){$a_2$};

    \draw[dashed] (0.5,1) -- (0.5,2);
    \draw[dashed] (0,1.5) -- (1,1.5);

    \draw[dashed] (1.5,0) -- (1.5,1);
    \draw[dashed] (1,0.5) -- (2,0.5);

    \node[below] at (1,0){(c) $\T_{a_2}$ and $\check \T_{a_2}$};

    \draw[line width=1mm] (0.5,0) -- (2,0) -- (2,1.5) -- (1.5,1.5) -- (1.5,2) -- (0,2) -- (0,0.5) -- (0.5,0.5) -- (0.5,0) -- (2,0);
    \draw[line width=0.5mm,white] (0.5,0) -- (2,0) -- (2,1.5) -- (1.5,1.5) -- (1.5,2) -- (0,2) -- (0,0.5) -- (0.5,0.5) -- (0.5,0) -- (2,0);
  \end{tikzpicture}\quad
  \begin{tikzpicture}[scale=3]
    \draw (0,0.5) rectangle (0.5,1);
    \draw (0,0.25) rectangle (0.25,0.5);
    \draw (0.25,0.25) rectangle (0.5,0.5);
    \node at (0,0.5){$\bullet$};
    \draw[dashed] (0,0.75) -- (0.5,0.75);
    \draw[dashed] (0.25,0.5) -- (0.25,1);
    \node[above right] at (0,0.5){$a_3$};

    \draw[line width=1mm] (0,0.25) -- (0.5,0.25) -- (0.5,1) -- (0,1) -- (0,0.25) -- (0.5, 0.25);
    \draw[line width=0.5mm,white] (0,0.25) -- (0.5,0.25) -- (0.5,1) -- (0,1);

    \node[below] at (0.25,0){(d) $\T_{a_3}$ and $\check \T_{a_3}$};
  \end{tikzpicture}
  \caption{Example refined patches. (a) Example quadrangulation with three vertices. (b) Regular patch $\T_{a_1}$ of interior vertex $a_1 \in \V_{\T}^{\irr}$ that equals its smallest regular refinement $\check \T_{a_1}$. (c) 1-irregular patch $\T_{a_2}$ of $a_2 \in \V_{\T}^{\irr}$; $\check \T_{a_2}$ is a refinement of $\T_{a_2}$ denoted by dashed lines. (d) Patch of boundary vertex $a_3 \in \V_\T^{\ext}$. The thick black line indicates edges in $\hE_a^{\ext,D}$; the double line edges in $\hE_a^{\ext,N}$.}
  \label{fig:exquadrangulation}
\end{figure}

Given $\T_a$, define the \emph{refined patch} $\hta$ as the smallest regular red-refinement of $\T_a$, and let each $\hT \in \hta$ inherit its local degree $p_{\hT}$ from its parent in $\T_a$. The key insight of considering the regular refinement $\hta$ instead of $\T_a$ was proposed in~\cite{Dolejsi2016} and allows us to write the discrete residual below as a sum of inner products with local polynomials.

For the edge sets $\E^{\irr}_a, \E^{\ext}_a, \E^{\ext,N}_a, \E^{\ext,D}_a$, define their $\check{\,}$-variants as the set of children edges; e.g., $\hE^{\irr}_a := \{\he \in \E_{\hta}: \exists e \in \E^{\irr}_a \text{~~s.t.~~} \he \subset e\}$.
See Figure~\ref{fig:exquadrangulation} for a few examples.

\section{Reducing the contraction problem to local saturation problems}
\label{sec:reduce}
This section will follow the same general structure of~\cite[\S3--4]{Canuto2017}; proofs are omitted for brevity but follow analogously to their counterpart in~\cite{Canuto2017}.

For $\omega$ a proper subset of $\overline \Omega$, let $\ltwoinp{\cdot}{\cdot}{\omega}$ denote the $L_2(\omega)$- or $[L_2(\omega)]^2$-inner product, and $\norm{\cdot}_{\omega}$ its norm. Unless mentioned otherwise, closed subspaces of $H^1(\omega)$ on which $\norm{\nabla \cdot}_{\omega}$ is equivalent to $\norm{\cdot}_{H^1(\omega)}$ are equipped with the $H^1(\omega)$-seminorm $\seminorm{\cdot}_\omega := \norm{\nabla \cdot}_{\omega}$.

\subsection{Residual}
For $\he \in \hE_a^{\irr}$, we denote with $\jump{\cdot}$ the jump operator and with $\vec n_{\he}$ a unit normal vector of $\he$.
We then define the global and localized residuals as
\[
  r_\T(v) := \ltwoinp{f}{v}{\Omega} - \honeinp{u_\T}{v}{\Omega}, \quad r_a(v) := r_\T(\psi_a v) \quad (v \in H^1(\Omega)),
\]
and observe that after integration by parts, the localized residual satisfies
\[
  r_a(v) = \sum_{\hT \in \hta} \ltwoinp{\psi_a (f + \laplace u_\T)}{v}{\hT} + \sum_{\he \in \hE_a^{\irr}} \ltwoinp{\psi_a \jump{\nabla u_\T \cdot \vec n_{\he}}}{v}{\he}.
\]
The following shows that the norms $\norm{r_a}_{H^1_*(\omega_a)'}$ may be used as a posteriori error indicators.
\begin{prop}[{Reliability and efficiency~\cite[Prop.~3.1]{Canuto2017}}]
  \label{prop:eff}
  There is a constant $C_{\eff} > 0$ with
  \[
    \seminorm{u - u_\T}^2_{\Omega} \leq 3 \sum_{a \in \V_\T} \norm{r_a}^2_{H^1_*(\omega_a)'}, \quad \norm{r_a}_{H^1_*(\omega_a)'} \leq C_{\eff} \seminorm{u - u_\T}_{\omega_a} ~~ (a \in \V_\T).
  \]
\end{prop}

\subsection{Data oscillation and discrete residual}
For a rectangle $T$, define $\Pi_p^T$ as the $L_2(T)$-orthogonal projection onto $\Q_p(T)$. The approximation $\Pi_{\hta} f$ to $f$ is then piecewise defined through
  $(\Pi_{\hta} f)|_{\hT} := \Pi^{\hT}_{p_{\hT}} f|_{\hT}$.
The difference between $f$ and its approximation is quantified by the \emph{data oscillation}, defined as
\[
  \osc(f, \T)^2 := \sum_{\hT \in \hta} h_{\hT}^2 \norm{f - \Pi^{\hT}_{p_{\hT}}}^2_{\hT}.
\]
We will study the \emph{discrete residual}, computed on discrete data $\Pi_{\hta} f$ instead of $f$:
\begin{equation}
  \dr_a(v) := \sum_{\hT \in \hta} \ltwoinp[s]{\phi_{\hT}}{v}{\hT} + \sum_{\he \in \hE_a^{\irr}} \ltwoinp[s]{\phi_{\he}}{v}{\he} \quad (v \in H^1(\omega_a))
  \label{eqn:discres}
\end{equation}
where
\[
  \phi_{\hT} := \psi_a (\Pi^{\hT}_{p_{\hT}} f + \laplace u_\T) \in \Q_{p_{\hT}+1}(\hT), \quad \text{and} \quad \phi_{\he} := \psi_a \jump{\nabla u_\T \cdot \vec n_{\he}}\in \P_{p_a+1}(\he).
\]
\begin{prop}[{Residual discrepancy~\cite[Cor.~3.4]{Canuto2017}}]
  \label{prop:discrep}
  There is a constant $C_{\osc} > 0$ with
  \[
    \abs{\sqrt{\sum_{a \in \V_\T} \norm{\dr_a}^2_{H^1_*(\omega_a)'}} - \sqrt{\sum_{a \in \V_\T} \norm{r_a}^2_{H^1_*(\omega_a)'}}} \leq C_{\osc} \osc(f, \T).
  \]
\end{prop}

\subsection{A theoretical AFEM}
We envision an abstract adaptive FEM that loops
\[
  \text{SOLVE -- ESTIMATE -- MARK -- REFINE},
\]
driven by the vertex-based a posteriori error indicators $\norm{\dr_a}_{H^1_*(\omega_a)'}$. The following result provides sufficient conditions for $p$-robust contraction of the error in energy norm. This AFEM can serve as an efficient error reducer in an instance-optimal $hp$-AFEM through a coarsening step; cf~\cite{Binev2015}.
\begin{prop}[{Contraction of AFEM~\cite[Prop.~4.1]{Canuto2017}}]
  Let $\theta \in (0,1]$ and $\rho \in [1,\infty)$ be constants. Suppose that for some $\lambda \in (0, \tfrac{\theta}{C_{\osc} \rho})$, we have
  \begin{enumerate}[(a)]
    \item small data oscillation: 
      \[
        \osc(f, \T) \leq \lambda \sqrt{\sum_{a \in \V_\T} \norm{\dr_a}_{H^1_*(\omega_a)'}^2},
      \]
    \item D\"orfler marking: a set $\mathcal M \subset \V_\T$ of marked vertices satisfying
      \[
        \sqrt{\sum_{a \in \mathcal M} \norm{\dr_a}_{H^1_*(\omega_a)'}^2} \geq \theta \sqrt{\sum_{a \in \V_{\T}} \norm{\dr_a}_{H^1_*(\omega_a)'}^2},
      \]
    \item local saturation: a closed subspace $\overline U \supset U_\T$ of $H^1_0(\Omega)$ that \emph{saturates} each residual dual norm:
      \[
        \norm{\dr_a}_{H^1_*(\omega_a)'} \leq \rho \norm{\dr_a}_{[H^1_*(\omega_a) \cap \overline U|_{\omega_a}]'} \quad (a \in \mathcal M).
      \]
  \end{enumerate}
  Then, with $\overline u \in \overline U$ the Galerkin approximation of the solution $u$ of~\eqref{eqn:poisson}, we have contraction,
  \[
    \seminorm{u - \overline u}_\Omega \leq \alpha \seminorm{u - u_\T}_\Omega, \quad \text{where} \quad \alpha = \alpha(\theta, \rho, \lambda) := \sqrt{1 - \left(\frac{\theta - C_{\osc} \lambda\rho}{3 C_{\eff}(1 + C_{\osc} \lambda)\rho}\right)^2},
  \]
  meaning that the error is reduced by a factor $\alpha$, uniformly bounded away from 1.
\end{prop}
\begin{remark}
  Assumption (a) is usually satisfied~\cite[Rem.~4.2]{Canuto2017}, and the D\"orfler marking for (b) can be constructed by ordering vertices by $\norm{\dr_a}_{H^1_*(\omega_a)'}$, so we will focus on (c).

  Given a function $q: \N \to \N$ such that
  \[
    \norm{\dr_a}_{H^1_*(\omega_a)'} \leq \rho \norm{\dr_a}_{[H^1_*(\omega_a) \cap \Q_{q(p_a+1)+1}^{-1}(\hta)]'} \quad (a \in \mathcal M),
  \]
  then, (c) is satisfied for any $U_\T \subset \overline U \subset H^1_0(\Omega)$ with
  \[
    H^1_*(\omega_a) \cap \Q^{-1}_{q(p_a+1)+1}(\hta) \subset H^1_*(\omega_a) \cap \overline U|_{\omega_a} \quad (a \in \mathcal M).
  \]

  In Theorem~\ref{thm:reduction} below, we reduce existence of $q$ to a small number of saturation problems on the reference square.
  Under this assumption, $\overline U$ can then be constructed as $U_{\check \T}$, where $\check \T$ is found through the following REFINE step:
  \begin{enumerate}[(i)]
    \item for each $a \in \mathcal M$, replace $\T_a$ by its smallest regular red-refinement $\hta$;
    \item for each $a \in \mathcal M$, for each $\hT \in \hta$, increase $p_\hT$ to $q(p_a + 1) + 1$;
    \item Take $\check \T$ as the smallest 1-irregular red-refinement of the resulting quadrangulation.
  \end{enumerate}

  The numerical results of \S\ref{sec:numres} suggest that the aforementioned reference problems are solved for $q(p) := p + \lceil \lambda p \rceil$ for any $\lambda > 0$. Each REFINE step multiplies the number of elements by not more than a factor $4$, and the local degrees by (up to) a constant factor $1 + \lceil \lambda \rceil$. Therefore, the dimension of the local finite element space is multiplied by not more than a factor $4(1 + \lceil \lambda \rceil)^2$; since the number of REFINE steps necessary for a fixed error reduction factor $\delta \in (0,1)$ is bounded by $M \leq \lceil \tfrac{\log \delta}{\log \alpha} \rceil$, this leads to an efficient error reducer.
\end{remark}

\subsection{Equivalent computable error quantities}
The localized discrete residuals $\dr_a$ provide, through their dual norms $\norm{\dr_a}_{H^1_*(\omega_a)'}$, reliable and efficient error indicators which can drive an AFEM. These dual norms are, however, not computable.

  For $p \geq 0$ and a rectangle $T$, the \emph{Raviart-Thomas space} of degree $p$ is defined as
  \[
    \RT_p(T) := \Q_{p+1,p}(T) \times \Q_{p,p+1}(T) \subset \vec H(\div; T).
  \]
The following two results underline the importance of this space for $p$-robust analysis.
\begin{lemma}[{$p$-robust inverse of divergence~\cite[Thm.~5]{Braess2009b}}]
  \label{lem:inverse}
  Let $T$ be a rectangle.  For $\varphi \in \Q_p(T)$, there is a $\vec \sigma \in \RT_p(T)$ with
  \[
    \div \vec \sigma = \varphi, \quad \norm{\vec \sigma}_{T} \lesssim \norm{\ltwoinp{\varphi}{\cdot}{T}}_{H^1_0(T)'}.
  \]
\end{lemma}
\begin{lemma}[{$p$-robust Raviart-Thomas extension~\cite[Cor.~3.4]{Costabel2010}}]
  \label{lem:rtext}
  Let $T$ be a rectangle with edges $\set{e_1, e_2, e_3, e_4}$.  Given $\varphi \in L_2(\partial T)$ such that
  \[
    \ltwoinp[s]{\varphi}{\bbone}{\partial T} = 0, \quad \text{and} \quad \varphi|_{e_i} \in \P_p(e_i) \quad (i \in \set{1,\ldots, 4})
  \]
  then there is a $\vec \sigma \in \RT_p(T)$ with
  \[
    \div \vec \sigma = 0, \quad \vec \sigma \cdot \vec n_{T} = \varphi~~\text{on $\partial T$}, \quad \norm{\vec \sigma}_{T} \lesssim \inf_{\set{\vec \tau \in \vec H(\div; T): \div \vec \tau = 0, \vec \tau \cdot \vec n_{T} = \varphi}} \norm{\vec \tau}_{T}.
  \]
\end{lemma}
In~\cite{Dolejsi2016}, Dolej\v{s}\'i \emph{et al.}~use these two lemmas to find a Raviart-Thomas flux $\vec \sigma_a \in \prod_{\hT \in \hta} \RT_{p_a}(\hT)$ with $p$-robust norm equivalence $\norm{\vec \sigma_a}_{\omega_a} \eqsim \norm{\dr_a}_{H^1_*(\omega_a)'}$, and present an efficient algorithm for its construction. The error indicators $\norm{\vec \sigma_a}_{\omega_a}$ \emph{can} be computed, and can therefore drive an AFEM.

\section{Reducing local saturation problem to reference saturation problems}
In this section, we prove the main theorem of the present work, reducing the local $p$-robust saturation problem to a small number of saturation problems on the reference square.

\subsection{Saturation coefficients}
  Let $\refsq := [-1,1]^2$ be the reference square. Given
    a closed linear subspace $\hat{\sp H} \subset H^1(\refsq)$ on which the $H^1(\refsq)$-seminorm is a norm,
    a finite-dimensional linear subspace $\hat{\sp V} \subset \hat{\sp H}$, and
    a set of functionals $\hat{\sp F} \subset \hat{\sp H}'$,
  define the \emph{saturation coefficient}
\[
  S(\hat{\sp H}, \hat{\sp V}, \hat{ \sp F}) := \sup_{\hat F \in \hat{\sp F}} \frac{\norm{\hat F}_{\hat{\sp H}'}}{\norm{\hat F}_{\hat{\sp V}'}}
\]
  which, if bounded, shows that $\hat{\sp V}$ is a large enough subspace to \emph{saturate} $\hat{\sp H}$ over the set $\hat{\sp F}$.

\begin{lemma}[Saturation extends to rectangles]
  \label{lem:refsat}
  When $\T \in \mathbb T$, then for any $T \in \T$,
  \[
    \sup_{F \in \sp F} \frac{\norm{F}_{\sp H'}}{\norm{F}_{\sp V'}} \lesssim \kappa_2(\mat B) S(\hat{\sp H}, \hat{\sp V}, \hat{ \sp F})
  \]
  where $F_T(\vec x) := \mat B \vec x + \vec b$ is an affine mapping from $T$ to $\refsq$, and $\sp H, \sp V, \sp F$ are determined by the pull-back, pull-back, and push-forward, respectively (cf.~\cite[p.82]{Brenner2008}).

  In words, saturation on the reference square extends to uniformly shape regular rectangles.
\end{lemma}

\subsection{Enumerating the interior edges of a refined patch}
\label{sub:edgetrav}
Refined patches will play an integral role in the proof of the forthcoming Theorem. Take $a \in \V_\T$, and let $\hta$ be its refined patch. We will construct an enumeration of the interior edges $\hE^{\irr}_a$ of $\hta$ as $(\he_i)_{i=1}^{n_a}$, where $n_a := \# \hE^{\irr}_a$, and for each interior edge, choose a specific square $\hT_i \in \hta$ adjacent to $\he_i$.

  Because every patch $\T_a$ is a 1-irregular collection of axis-aligned rectangles, there is only a finite number of different refined patch types. In fact, it can be shown that up to rotation/flipping of $\hta$, all patches fall in one of the 13 types on the right of Figure~\ref{fig:edgetrav}.
  
  Overlay the vertex $a$ with the $\hat a$ in the $4 \times 4$ grid to the left of Figure~\ref{fig:edgetrav}. Then every $\he \in \hE^{\irr}_a$ inherits a number $1 \leq k(i) \leq 24$ from the grid. We then enumerate $(\he_i)_{i=1}^{n_a}$ in increasing order of the values $k(i)$, and we choose $\hT_i$ as the square above or to the left of $\he_i$ (whichever is applicable).

\begin{figure}[t!]
  \begin{quote}
\begin{minipage}[t][5.1cm]{5.1cm}
  \begin{tikzpicture}[scale=1.25]
    \draw[step=1cm,draw] (0,0) grid (4,4);
    \node[above] at (3.5,1){$1$};
    \node[above] at (3.5,2){$2$};
    \node[above] at (3.5,3){$3$};
    \node[left] at (3,0.5){$4$};
    \node[left] at (2,0.5){$5$};
    \node[left] at (1,0.5){$6$};

    \node[left] at (3,1.5){$7$};
    \node[left] at (3,2.5){$8$};
    \node[left] at (3,3.5){$9$};
    \node[above] at (2.5,1){$10$};
    \node[above] at (1.5,1){$11$};
    \node[above] at (0.5,1){$12$};

    \node[above] at (2.5,2){$13$};
    \node[above] at (2.5,3){$14$};
    \node[left] at (2,1.5){$15$};
    \node[left] at (1,1.5){$16$};
    \node[left] at (2,2.5){$17$};
    \node[left] at (2,3.5){$18$};

    \node[above] at (1.5,2){$19$};
    \node[above] at (0.5,2){$20$};
    \node[above] at (1.5,3){$21$};
    \node[left] at (1,2.5){$22$};
    \node[left] at (1,3.5){$23$};
    \node[above] at (0.5,3){$24$};

    \node[below right] at (2,2){$\hat a$};
    \node at (2,2){$\bullet$};

    \draw[line width=1mm] (0,0) -- (4,0) -- (4,4) -- (0,4) -- (0,0) -- (4,0);
    \draw[line width=0.5mm,white] (0,0) -- (4,0) -- (4,4) -- (0,4) -- (0,0) -- (4,0);
  \end{tikzpicture}
\end{minipage}
\begin{minipage}[b][5.1cm]{\linewidth - 5.1cm}
  \begin{tikzpicture}[scale=0.45]
    \draw (0,2) -- (4,2) (2,0) -- (2,4);

    \draw (0,1) -- (2,1) (1,0) -- (1,2);
    \draw[dashed] (0,3) -- (2,3) (1,2) -- (1,4);
    \draw[dashed] (2,3) -- (4,3) (3,2) -- (3,4);
    \draw[dashed] (2,1) -- (4,1) (3,0) -- (3,2);

    \node[below right] at (2,2){$a$};
    \node at (2,2){$\bullet$};

    \draw[line width=1mm] (1,0) -- (4,0) -- (4,4) -- (0,4) -- (0,1) -- (1,1) -- (1,0) -- (4,0);
    \draw[line width=0.5mm,white] (1,0) -- (4,0) -- (4,4) -- (0,4) -- (0,1) -- (1,1) -- (1,0) -- (4,0);
  \end{tikzpicture}%
\hspace{0.3em}%
  \begin{tikzpicture}[scale=0.45]
    \draw (0,2) -- (4,2) (2,0) -- (2,4);

    \draw (0,1) -- (2,1) (1,0) -- (1,2);
    \draw[dashed] (0,3) -- (2,3) (1,2) -- (1,4);
    \draw (2,3) -- (4,3) (3,2) -- (3,4);
    \draw[dashed] (2,1) -- (4,1) (3,0) -- (3,2);

    \node[below right] at (2,2){$a$};
    \node at (2,2){$\bullet$};

    \draw[line width=1mm] (1,0) -- (4,0) -- (4,3) -- (3,3) -- (3,4) -- (0,4) -- (0,1) -- (1,1) -- (1,0) -- (4,0);
    \draw[line width=0.5mm,white] (1,0) -- (4,0) -- (4,3) -- (3,3) -- (3,4) -- (0,4) -- (0,1) -- (1,1) -- (1,0) -- (4,0);
  \end{tikzpicture}%
\hspace{0.3em}%
  \begin{tikzpicture}[scale=0.45]
    \draw (0,2) -- (3,2) (2,0) -- (2,4);

    \draw[dashed] (0,1) -- (2,1) (1,0) -- (1,2);
    \draw[dashed] (0,3) -- (2,3) (1,2) -- (1,4);
    \draw (2,3) -- (3,3) (3,2) -- (3,4);
    \draw (2,1) -- (3,1) (3,0) -- (3,2);

    \node[below right] at (2,2){$a$};
    \node at (2,2){$\bullet$};

    \draw[line width=1mm] (0,0) -- (3,0) -- (3,4) -- (0,4) -- (0,0) -- (3,0);
    \draw[line width=0.5mm,white] (0,0) -- (3,0) -- (3,4) -- (0,4) -- (0,0) -- (3,0);
  \end{tikzpicture}%
\hspace{0.3em}%
  \begin{tikzpicture}[scale=0.45]
    \draw (0,2) -- (3,2) (2,0) -- (2,3);

    \draw[dashed] (0,1) -- (2,1) (1,0) -- (1,2);
    \draw (0,3) -- (2,3) (1,2) -- (1,3);
    \draw (2,3) -- (3,3) (3,2) -- (3,3);
    \draw (2,1) -- (3,1) (3,0) -- (3,2);

    \node[below right] at (2,2){$a$};
    \node at (2,2){$\bullet$};

    \draw[line width=1mm] (0,0) -- (3,0) -- (3,3) -- (0,3) -- (0,0) -- (3,0);
    \draw[line width=0.5mm,white] (0,0) -- (3,0) -- (3,3) -- (0,3) -- (0,0) -- (3,0);
  \end{tikzpicture}%
\hspace{0.3em}%
  \begin{tikzpicture}[scale=0.45]
    \draw (2,1) -- (2,3) (1,2) -- (3,2);

    \node[below right] at (2,2){$a$};
    \node at (2,2){$\bullet$};

    \draw[line width=1mm] (1,1) -- (1,3) -- (3,3) -- (3,1) -- (1,1) -- (1,3);
    \draw[line width=0.5mm,white] (1,1) -- (1,3) -- (3,3) -- (3,1) -- (1,1) -- (1,3);

    \draw[white, line width=1mm] (2,0) -- (3,0);
  \end{tikzpicture}
  \hfill
  \vfill
  \begin{tikzpicture}[scale=0.45]
    \draw (0,2) -- (4,2) (2,0) -- (2,4);

    \draw[dashed] (0,1) -- (2,1) (1,0) -- (1,2);
    \draw[dashed] (2,3) -- (4,3) (3,2) -- (3,4);
    \draw (2,1) -- (4,1) (3,0) -- (3,2);

    \node[below right] at (2,2){$a$};
    \node at (2,2){$\bullet$};

    \draw[line width=1mm] (0,2) -- (0,0) -- (3,0) -- (3,1) -- (4,1) -- (4,4) -- (2,4) -- (2,2) -- (0,2) -- (0,0);
    \draw[line width=0.5mm,white] (0,2) -- (0,0) -- (3,0) -- (3,1) -- (4,1) -- (4,4) -- (2,4);
  \end{tikzpicture}%
\hspace{0.3em}%
  \begin{tikzpicture}[scale=0.45]
    \draw (1,1) -- (2,1) (2,0) -- (2,2) -- (4,2);
    \draw[dashed] (2,1) -- (4,1) (2,3) -- (4,3) (3,0) -- (3,4);

    \node[below right] at (2,2){$a$};
    \node at (2,2){$\bullet$};

    \draw[line width=1mm] (1,2) -- (1,0) -- (4,0) -- (4,4) -- (2,4) -- (2,2) -- (1,2) -- (1,0);
    \draw[line width=0.5mm,white] (1,2) -- (1,0) -- (4,0) -- (4,4) -- (2,4);
  \end{tikzpicture}%
\hspace{0.3em}%
  \begin{tikzpicture}[scale=0.45]
    \draw (2,0) -- (2,2) -- (3,2) (2,1) -- (3,1);
    \draw[dashed] (1,0) -- (1,2) (0,1) -- (2,1);

    \node[below right] at (2,2){$a$};
    \node at (2,2){$\bullet$};

    \draw[line width=1mm] (0,2) -- (0,0) -- (3,0) -- (3,3) -- (2,3) -- (2,2) -- (0,2) -- (0,0);
    \draw[line width=0.5mm,white] (0,2) -- (0,0) -- (3,0) -- (3,3) -- (2,3);
  \end{tikzpicture}%
\hspace{0.3em}%
  \begin{tikzpicture}[scale=0.45]
    \draw (1,1) -- (2,1) (2,0) -- (2,2) -- (4,2) (3,2) -- (3,3);
    \draw[dashed] (3,0) -- (3,2) (2,1) -- (4,1);

    \node[below right] at (2,2){$a$};
    \node at (2,2){$\bullet$};

    \draw[line width=1mm] (1,2) -- (1,0) -- (4,0) -- (4,3) -- (2,3) -- (2,2) -- (1,2) -- (1,0);
    \draw[line width=0.5mm,white] (1,2) -- (1,0) -- (4,0) -- (4,3) -- (2,3);
  \end{tikzpicture}%
\hspace{0.3em}%
  \begin{tikzpicture}[scale=0.45]
    \draw (2,2) -- (4,2) (3,3) -- (3,2);
    \draw[dashed] (3,0) -- (3,2) (2,1) -- (4,1);

    \node[below right] at (2,2){$a$};
    \node at (2,2){$\bullet$};

    \draw[line width=1mm] (2,0) -- (4,0) -- (4,3) -- (2,3) -- (2,0) -- (4,0);
    \draw[line width=0.5mm,white] (2,0) -- (4,0) -- (4,3) -- (2,3);
  \end{tikzpicture}
  \hfill
  \vfill
  {\color{white}.}
  \hfill
  \begin{tikzpicture}[scale=0.45]
    \draw (2,1) -- (2,2) -- (3,2);

    \node[below right] at (2,2){$a$};
    \node at (2,2){$\bullet$};

    \draw[line width=1mm] (1,2) -- (1,1) -- (3,1) -- (3,3) -- (2,3) -- (2,2) -- (1,2) -- (1,1);
    \draw[line width=0.5mm,white] (1,2) -- (1,1) -- (3,1) -- (3,3) -- (2,3);

  \end{tikzpicture}%
\hspace{0.3em}%
  \begin{tikzpicture}[scale=0.45]
    \draw (2,2) -- (3,2);

    \node[below right] at (2,2){$a$};
    \node at (2,2){$\bullet$};

    \draw[line width=1mm] (2,1) -- (3,1) -- (3,3) -- (2,3) -- (2,1) -- (3,1);
    \draw[line width=0.5mm,white] (2,1) -- (3,1) -- (3,3) -- (2,3);

  \end{tikzpicture}%
\hspace{0.3em}%
  \begin{tikzpicture}[scale=0.45]
    \node[below right] at (2,2){$a$};
    \node at (2,2){$\bullet$};

    \draw[line width=1mm] (2,1) -- (3,1) -- (3,2) -- (2,2) -- (2,1) -- (3,1);
    \draw[line width=0.5mm,white] (2,1) -- (3,1) -- (3,2);

  \end{tikzpicture}
  \hfill
  {\color{white}.}
\end{minipage}
\end{quote}
  \caption{Left: a $4 \times 4$ grid with vertex $\hat a$, and an enumeration of its interior edges. Right: the 13 fundamentally different refined patch types, with the double line indicating Neumann edges $\hE^{\ext,N}_a$ of the patch boundary, and the thick black line Dirichlet edges $\hE^{\ext,D}_a$. We enumerate interior edges of a patch by overlaying its vertex $a$ with $\hat a$ in the left grid, and numbering them in increasing order.}
  \label{fig:edgetrav}
\end{figure}
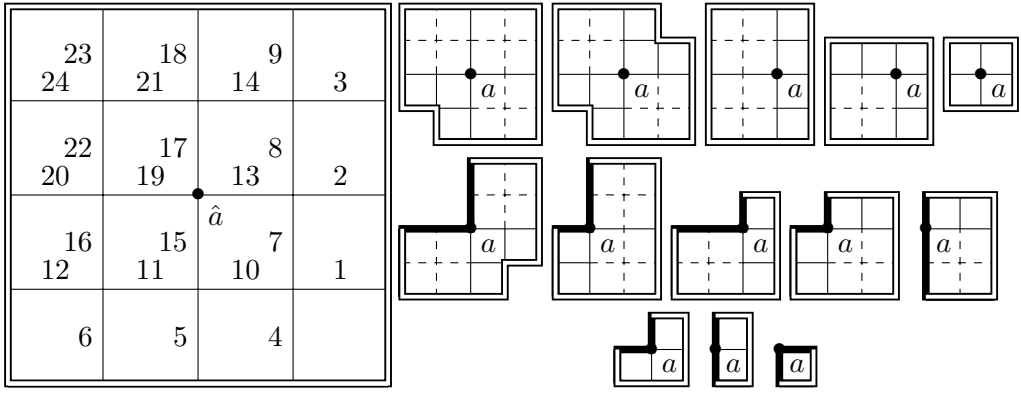

\subsection{Main theorem}
Let $T$ be a rectangle. When $\gamma \subset \partial T$ with $\meas(\gamma) > 0$, the space $H^1_{0, \gamma}(T)$ denotes the closure in $H^1(T)$ of the smooth functions on $\overline T$ that vanish on $\gamma$. By abuse of notation, when $\E = \set{\gamma}$ is a collection of such parts of the boundary, $H^1_{0, \E}(T)$ will denote the closure of smooth functions that vanish on every $\gamma$ separately.

For brevity purposes, write the restriction of $H^1_*(\omega_a)$ to piecewise polynomials as
\[
  \hprime{p} := H^1_*(\omega_a) \cap \Q^{-1}_p(\hta) \quad (p \in \N, ~~ a \in \V_\T).
\]

  We enumerate the edges of the reference square $\refsq$ as \mbox{$\E_{\refsq} = (\hat e_1, \hat e_2, \hat e_3, \hat e_4)$}, in counterclockwise fashion, starting from the rightmost edge.
\begin{theorem}[Reduction of $p$-robust saturation]
  Given the following sets of subsets of $\E_{\refsq}$,
  \begin{equation}
      \mathbb E^{(A)} := \set{\E \subset \E_{\refsq}: \E \not= \emptyset}, \quad
    \mathbb E^{(B)} := \set{ \set{\hat e_2}, \set{\hat e_3}, \set{\hat e_2, \hat e_3}, \set{\hat e_2, \hat e_3, \hat e_4} },
    \label{eqn:Es}
  \end{equation}
  define the following reference saturation coefficients
  \begin{align*}
    S^{(A)}_{\E, p, q} &:= S\left(H^1_{0, \E}(\refsq), ~~ H^1_{0, \E}(\refsq) \cap \Q_q(\refsq), ~~ \{h \mapsto \ltwoinp{\phi}{h}{\refsq} : \phi \in \Q_p(\refsq)\} \right) \quad \left(\E \in \mathbb E^{(A)}\right), \\
    S^{(B)}_{\E, p, q} &:= S\left(H^1_{0, \E}(\refsq), ~~ H^1_{0, \E}(\refsq) \cap \Q_q(\refsq), ~~ \{ h \mapsto \ltwoinp[s]{\phi}{h}{\hat e_1} : \phi \in \P_p(\hat e_1)\} \right) \quad \left(\E \in \mathbb E^{(B)}\right), \\
    S^{(C)}_{p, q} &:= S\left(H^1(\refsq)/\R, ~~ \Q_q(\refsq)/\R, ~~ \{h \mapsto \ltwoinp[s]{\phi}{h}{\hat e_1} : \phi \in \P_p(\hat e_1)/\R\} \right).
  \end{align*}
  If for some function $q: \N \to \N$, it holds that
  \begin{align*}
    \hat S := &\sup_{p} \max \set{S^{(A)}_{\E, p, q(p)}: \E \in \mathbb E^{(A)}} \cup \set{S^{(B)}_{\E, p, q(p)}: \E \in \mathbb E^{(B)}} \cup \set{S^{(C)}_{p, q(p)}} < \infty,
  \end{align*}
  then we have $p$-robust saturation, in that
  \begin{equation}
    \norm{\dr_a}_{H^1_*(\omega_a)'} \lesssim \norm{\dr_a}_{\hprime{q(p_a+1)+1}'}
    \label{eqn:reduction}
  \end{equation}
  dependent on $\hat S$, but independent of the quadrangulation $\T$ and its local degrees.
  \label{thm:reduction}
\end{theorem}

\subsection*{Outline of proof}
Our proof is similar in taste to \cite[Thm.~7.1]{Canuto2017}, with some details requiring a different approach.
We will perform three steps. Write, as in~\eqref{eqn:discres},
\[
  \dr_a(v) = \sum_{\hT \in \hta} \ltwoinp{\phi_{\hT}}{v}{\hT} + \sum_{\he \in \hE_a^{\irr}} \ltwoinp[s]{\phi_{\he}}{v}{\he} \quad (v \in H^1(\omega_a))
\]
for some $\phi_{\hT} \in \Q_{p_{\hT}+1}(\hT)$ and $\phi_{\he} \in \P_{p_a+1}(\he)$.
In \textit{Step (A)} below, we bound the dual norm of the set of element terms;
in \textit{Steps (B) and (C)}, we do the same for the set of edge terms.
Throughout the proof, we will use assumption $p_{\hT} \geq 1$ to find that, for interior vertices $a \in \V_\T^{\irr}$, the residual vanishes on constants ($\psi_a \in U_\T$ so $\dr_a(\bbone) = \dr(\psi_a \bbone) = \dr(\psi_a) = 0$).

In \textit{Step (A)}, we employ $\sup_p \max_{\E \in \mathbb E^{(A)}} S^{(A)}_{\E, p, q(p)} < \infty$ to find, on every rectangle $\hT \in \hta$, a functional $\dr_{\hT} \in H^1_*(\omega_a)'$ with
\begin{equation}
  \norm{\dr_{\hT}}_{H^1_*(\omega_a)'} \lesssim \norm{\dr_a}_{\hprime{q(p_a+1)}'}\quad\text{and}\quad \dr_{\hT}(\bbone) = 0,
  \label{eqn:localnormbound}
\end{equation}
that removes the $\hT$-contribution from $\dr_a$, in the sense that the residual
\[
  \dr^{(0)}_a := \sum_{\hT \in \hta} \dr_{\hT}
\]
satisfies, for $v \in H^1_*(\omega_a)$,
\begin{equation}
  \dr_a(v) - \dr^{(0)}_a(v) = \sum_{\he \in \hE_a^{\irr}} \ltwoinp[s]{\phi^{(0)}_{\he}}{v}{\he} \quad \text{for some} \quad \phi^{(0)}_{\he} \in \P_{p_a+1}(\he).
  \label{eqn:r0props}
\end{equation}

Next, in \textit{Step (B)}, we use the enumeration $(\he_j)_{j=1}^{n_a}$ of the interior edges $\hE_a^{\irr}$ where $n_a := \# \hE_a^{\irr}$. At step $i \in \set{1, \ldots, n_a-1}$, we use $\sup_p \max_{\E \in \mathbb E^{(B)}} S^{(B)}_{\E, p, q(p)} < \infty$ and Lemma~\ref{lem:stepb} below to find a functional $\dr_{\he_i} = \dr_{\hT_i, \he_i} \in H^1_*(\omega_a)'$ with
\begin{equation}
  \norm{\dr_{\he_i}}_{H^1_*(\omega_a)'} \lesssim \norm{\dr_a}_{\hprime{q(p_a+1)+1}'} \quad\text{and}\quad
  \dr_{\he_i}(\bbone) = 0,
  \label{eqn:rinormbound}
\end{equation}
that removes the $\he_i$-contribution from $\dr^{(i-1)}_a$ while not re-introducing contributions on edges $\he_j$ for $j < i$, in the sense that the residual
\[
  \dr^{(i)}_a := \dr^{(i-1)}_a + \dr_{\he_i}
\]
satisfies, for $v \in H^1_*(\omega_a)$,
\begin{equation}
  \dr_a(v) - \dr^{(i)}_a(v) = \sum_{j \geq i + 1} \ltwoinp[s]{\phi_{\he_j}^{(i)}}{v}{\he_j} \quad \text{for some} \quad \phi_{\he_j}^{(i)} \in \P_{p_a+1}(\he_j).
  \label{eqn:riprop}
\end{equation}

Lastly, in \textit{Step (C)}, the final iteration $i = n_a$, we make a distinction. When $a \in \V_\T^{\ext}$ is a boundary vertex, we construct a $\dr_{\he_{n_a}} \in H^1_*(\omega_a)'$ for which~\eqref{eqn:rinormbound} and~\eqref{eqn:riprop} hold once more. 
Then through the triangle inequality, $\# \hta \leq 16$, and $\# \hE^{\irr}_a \leq 24$ we find
\[
  \norm{\dr_a}_{H^1_*(\omega_a)'} \leq 
  \sum_{\hT \in \hta} \norm{\dr_{\hT}}_{H^1_*(\omega_a)'} +
  \sum_{j=1}^{n_a} \norm{\dr_{\he_j}}_{H^1_*(\omega_a)'}
  \lesssim \norm{\dr_a}_{\hprime{q(p_a+1)+1}'}.
\]
When $a \in \V_\T^{\irr}$ is an interior vertex, we use $\sup_p S^{(C)}_{p,q(p)} < \infty$ to bound
\begin{equation}
  \norm{\dr_a - \dr_a^{(n_a-1)}}_{H^1_*(\omega_a)'} \lesssim \norm{\dr_a}_{\hprime{q(p_a+1)}'}
  \label{eqn:drdiffbound}
\end{equation}
which implies that
\[
  \norm{\dr_a}_{H^1_*(\omega_a)'} \leq 
  \norm{\dr_a - \dr_a^{(n_a - 1)}}_{H^1_*(\omega_a)'} +
  \sum_{\hT \in \hta} \norm{\dr_{\hT}}_{H^1_*(\omega_a)'} +
  \sum_{j=1}^{n_a-1} \norm{\dr_{\he_j}}_{H^1_*(\omega_a)'}
  \lesssim \norm{\dr_a}_{\hprime{q(p_a+1)+1}'}.
\]
In either case, we conclude that~\eqref{eqn:reduction} must hold.

\subsection*{Extension lemma}
Proving, in particular, inequality~\eqref{eqn:rinormbound} requires some creativity. Assume for now that $a$ is a boundary vertex (the other case is handled in the main proof). We will require the intermediate result that for some specific finite-dimensional subspace of polynomials $\sp V_i \subset H^1(\hT_i)$, there is, for each $v \in \sp V_i$, a piecewise polynomial $Ev \in H^1_*(\omega_a)$ with
\[
  \seminorm{Ev}_{\omega_a} \lesssim \seminorm{v}_{\hT_i}, \quad \text{and} \quad \ltwoinp{\phi^{(i-1)}_{\he_i}}{v}{\he_i} = \dr_a(Ev) - \dr^{(i-1)}_a(Ev).
\]
Our approach is the following.
Note that $\ltwoinp{\phi^{(i-1)}_{\he_i}}{v}{\he_i}$ is an inner product over a single edge, whereas $\dr_a(Ev) - \dr^{(i-1)}_a(Ev)$ is a sum of inner products $\ltwoinp{\phi^{(i-1)}_{\he_j}}{Ev}{\he_j}$ on all interior edges $\he_j$ with $j \geq i$ (see~\eqref{eqn:riprop}). The desired equality holds for all $v \in \sp V_i$ surely when $Ev$ \text{extends} $v$ (in that $Ev|_{\hT_i} = v$), and $Ev|_{\he_j} = 0$ for every $j \geq i+1$. Moreover, $a \in \V^{\ext}_{\T}$, so $Ev \in H^1_*(\omega_a)$ should vanish on all edges in $\hE^{\ext,D}_{a}$ as well. This gives rise to the \emph{set of patch (resp.~local) Dirichlet edges},
\begin{equation}
  \hE^D_{a,i} := \hE^{\ext,D}_a \cup \set{\he_j \in \hE_a^{\irr}: j \geq i+1}, \quad \hE^{\loc,D}_{a,i} := \hE^D_{a,i} \cap \hE_{\hT_i} \quad (i = 1, \ldots, n_a),
  \label{eqn:diredges}
\end{equation}
and for $v$ that vanishes on all local Dirichlet edges, $Ev$ then vanishes on all \emph{patch} Dirichlet edges $\he \in \hE^D_{a,i}$. Existence of this $Ev$ depends on the enumeration $(\he_i)_{i=1}^{n_a}$ of interior edges. The following lemma shows that with our particular construction, we can build a suitable $E$.

\begin{lemma}[Bounded polynomial extension]
  Let $\hta$ be one of the 13 refined patch types of Figure~\ref{fig:edgetrav}. Let $n_a$, $(\he_i)_{i=1}^{n_a}$, and $(\hT_i)_{i=1}^{n_a}$ be as defined in~\S\ref{sub:edgetrav}.
  For each square $\hT_i$, we enumerate its edges as $(e_1, e_2, e_3, e_4)$, in counterclockwise fashion, starting from the rightmost edge.

  For $1 \leq i \leq n_a-1$, and $i=n_a$ when $a$ is an external vertex, the following holds.
  \begin{enumerate}
    \item The set $\hE^{\loc,D}_{a,i}$ is nonempty. In fact, one of five situations occurs:
      \begin{gather*}
        \text{\emph{(a)}}~~ \hE^{\loc,D}_{a,i} = \set{e_1, e_2, e_3}, \quad \text{\emph{(b)}}~~\hE^{\loc,D}_{a,i} = \set{e_2, e_3, e_4}, \quad \text{\emph{(c)}}~~\hE^{\loc,D}_{a,i} = \set{e_2, e_3}, \\
        \text{\emph{(d)}}~~\hE^{\loc,D}_{a,i} = \set{e_2} \text{ and } e_3 \in \hE^{\ext,N}_a, \quad \text{\emph{(e)}}~~\hE^{\loc,D}_{a,i} = \set{e_3} \text{ and } e_2 \in \hE^{\ext,N}_a.
      \end{gather*}
    \item There is a bounded linear map $E: H^1_{0, \hE^{\loc,D}_{a,i}}(\hT_i) \cap \Q_{q(p_a+1)}(\hT_i) \to H^1(\omega_a) \cap \Q_{q(p_a+1)+1}^{-1}(\hta)$ so that for all $v$, its extension $Ev$ vanishes on patch Dirichlet edges; specifically,
      \[
        Ev|_{\hT_i} = v, \quad \seminorm{Ev}_{\omega_a} \lesssim \seminorm{v}_{\hT_i}, \quad Ev|_{\he} = 0 ~~ (\he \in \hE^D_{a,i}).
      \]
  \end{enumerate}
  \label{lem:stepb}
\end{lemma}

\begin{proof}
  A careful visual inspection of the enumeration for each of the 13 patch types of Figure~\ref{fig:edgetrav} shows that condition (1) holds: by enumerating the edges right-to-left, bottom-to-top, we ensure $e_2$ and $e_3$ are (situations (a--c)) both in $\hE^{\ext,D}_a$ or equal to some $\he_j$ for $j > i$, or (situations (d--e)) when $\hT_i$ is in the topmost row or leftmost column, either $e_2$ or $e_3$ is in $\hE^{\ext,N}_a$, but never both.

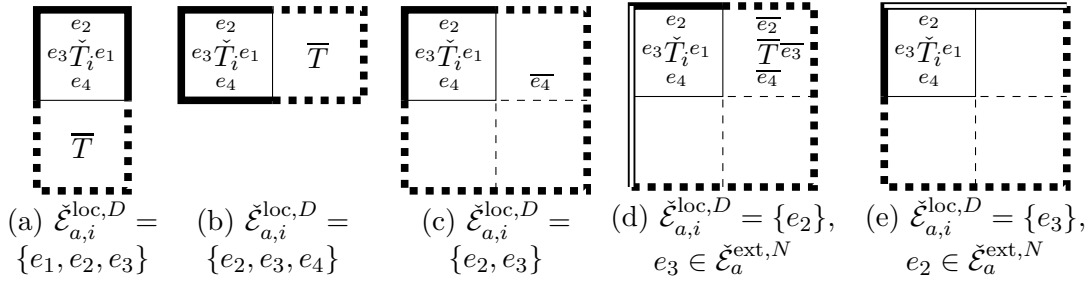
\begin{figure}[h!]
  \begin{tikzpicture}[scale=1.2]
    \draw (0,0) rectangle (1,1);

    \node at (0.5,0.5){$\hT_i$};
    \node[left] at (1,0.5){$_{e_1}$};
    \node[below] at (0.5,1){$_{e_2}$};
    \node[right] at (0,0.5){$_{e_3}$};
    \node[above] at (0.5,0){$_{e_4}$};

    \node at (0.5,-0.5){$\overline T$};

    \draw[line width=1mm] (0,0) -- (0,1) -- (1,1) -- (1,0);
    \draw[line width=1mm,dashed] (0,0) -- (0,-1) -- (1,-1) -- (1,0);
    
    \node[below,align=center] at (0.5,-1){(a) $\hE^{\loc,D}_{a,i} =$\\$\set{e_1, e_2, e_3}$};
  \end{tikzpicture}
  \begin{tikzpicture}[scale=1.2]
    \draw (0,0) rectangle (1,1);

    \node at (0.5,0.5){$\hT_i$};
    \node[left] at (1,0.5){$_{e_1}$};
    \node[below] at (0.5,1){$_{e_2}$};
    \node[right] at (0,0.5){$_{e_3}$};
    \node[above] at (0.5,0){$_{e_4}$};

    \node at (1.5,0.5){$\overline T$};

    \draw[line width=1mm] (1,1) -- (0,1) -- (0,0) -- (1,0);
    \draw[line width=1mm,dashed] (1,0) -- (2,0) -- (2,1) -- (1,1);
    
    \node[below,align=center] at (1,-1){(b) $\hE^{\loc,D}_{a,i} =$\\$\set{e_2, e_3, e_4}$};
  \end{tikzpicture}\hspace{0.8em}
  \begin{tikzpicture}[scale=1.2]
    \draw (0,0) rectangle (1,1);
    \draw[dashed] (1,-1) -- (1,0) -- (2,0);

    \node at (0.5,0.5){$\hT_i$};
    \node[left] at (1,0.5){$_{e_1}$};
    \node[below] at (0.5,1){$_{e_2}$};
    \node[right] at (0,0.5){$_{e_3}$};
    \node[above] at (0.5,0){$_{e_4}$};

    \node[above] at (1.5,0){$_{\overline{e_4}}$};

    \draw[line width=1mm] (1,1) -- (0,1) -- (0,0);
    \draw[line width=1mm,dashed] (0,0) -- (0,-1) -- (2,-1) -- (2,1) -- (1,1);
    
    \node[below,align=center] at (1,-1){(c) $\hE^{\loc,D}_{a,i} =$\\$\set{e_2, e_3}$};
  \end{tikzpicture}
  \begin{tikzpicture}[scale=1.2]
    \draw (0,0) rectangle (1,1);
    \draw[dashed] (1,-1) -- (1,0) -- (2,0);

    \node at (0.5,0.5){$\hT_i$};
    \node[left] at (1,0.5){$_{e_1}$};
    \node[below] at (0.5,1){$_{e_2}$};
    \node[right] at (0,0.5){$_{e_3}$};
    \node[above] at (0.5,0){$_{e_4}$};

    \node at (1.5,0.5){$\overline T$};
    \node[below] at (1.5,1){$_{\overline{e_2}}$};
    \node[left] at (2,0.5){$_{\overline{e_3}}$};
    \node[above] at (1.5,0){$_{\overline{e_4}}$};

    \draw[line width=1mm] (1,1) -- (0,1) -- (0,-1);
    \draw[line width=1mm,dashed] (0,-1) -- (2,-1) -- (2,1) -- (1,1);
    \draw[white,line width=0.5mm] (0,-1) -- (0,1);
    
    \node[below,align=center] at (1,-1){(d) $\hE^{\loc,D}_{a,i} = \set{e_2}$,\\$e_3 \in \hE^{\ext,N}_a$};
  \end{tikzpicture}
  \begin{tikzpicture}[scale=1.2]
    \draw (0,0) rectangle (1,1);
    \draw[dashed] (1,-1) -- (1,0) -- (2,0);

    \node at (0.5,0.5){$\hT_i$};
    \node[left] at (1,0.5){$_{e_1}$};
    \node[below] at (0.5,1){$_{e_2}$};
    \node[right] at (0,0.5){$_{e_3}$};
    \node[above] at (0.5,0){$_{e_4}$};

    \draw[line width=1mm] (2,1) -- (0,1) -- (0,0);
    \draw[line width=1mm,dashed] (0,0) -- (0,-1) -- (2,-1) -- (2,1);
    \draw[white,line width=0.5mm] (0,1) -- (2,1);
    
    \node[below,align=center] at (1,-1){(e) $\hE^{\loc,D}_{a,i} = \set{e_3}$,\\$e_2 \in \hE^{\ext,N}_a$};
  \end{tikzpicture}
  \caption{The five different extension cases of Lemma~\ref{lem:stepb}. The full thick line on $\partial \hT_i$ denotes its local Dirichlet boundary $\hE^{\loc,D}_{a,i}$, and the dashed thick line shows the Dirichlet boundary of the extension; double lines indicate edges in $\hE^{\ext,N}_a$.}
  \label{fig:b2ext}
\end{figure}
  
    By the first result of this Lemma, there are essentially five cases to look at. See Figure~\ref{fig:b2ext}. Denote with $\mathscr T$ the union of squares in the appropriate case.
    Let $v \in H^1_{0, \hE^{\loc,D}_{a,i}}(\hT_i) \cap \Q_{q(p_a+1)}(\hT_i)$. We will use multiple reflections of $v$ to find a piecewise polynomial $\underline v \in H^1(\mathscr T)$ (of degree $q(p_a+1)+1$) that vanishes on the part of $\partial \mathscr T$ denoted by the thick line. Restricting $\underline v$ to $\mathscr T \cap \omega_a$ (because $\mathscr T$ may contain squares outside $\hta$) yields a function that vanishes on the edges $\he \in \hE^{\irr}_a$ with $\he \subset \partial \mathscr T$, so that we can easily zero-extend $\underline v|_{\mathscr T \cap \omega_a}$ to $Ev \in H^1(\omega_a) \cap \Q_{q(p_a+1) + 1}^{-1}(\hta)$.
    
    The choice of $\hT_i$ ensures that $\he_i$ is its right or bottom edge. Moreover, the enumeration is bottom-right to top-left, so that every patch Dirichlet edge is positioned either above or to the left of $\hT_i$. On the other hand, the support of our extension $Ev$ is---as we will shortly see---to the right or bottom of $\hT_i$. Therefore, $Ev$ necessarily vanishes on all of $\hE^D_{a,i}$.
    
    It remains to construct $\underline v$ with the desired properties above, for each situation.
\begin{enumerate}[(a)]
  \item Denote with $\overline{v}, \overline{T}$ the reflections of $v$ and $\hT_i$ across $e_4$. Then $\overline v|_{e_4} = v|_{e_4}$ and $\seminorm{\overline v}_{\overline T} = \seminorm{v}_{\hT}$, so the extension $\underline v$ defined by
    $\underline v|_{\hT} := v$ and $\underline v|_{\overline T} := \overline v$
    vanishes on all of $\partial (\hT_i \cup \overline T)$, is continuous globally, and polynomial on both squares separately.
  \item The proof of this case is analogous to that of (a).
  \item Denote with $\overline{e_4}$ the reflection of $e_4$ across $e_1$. Denote with $\overline{\overline v}$ the extension of $v$ on $\hT_i \cup \overline T$. The extension $\underline v$ of $\overline{\overline v}$ across $e_4 \cup \overline{e_4}$ is the desired function.
  \item Denote with $\overline{v}, \overline{T}, \overline{e_2}, \overline{e_3}, \overline{e_4}$ the reflections of $v, T, e_2, e_3, e_4$ across $e_1$, respectively. Let $\overline \phi \in \Q_1(\overline T)$ be a decay function defined by $\overline \phi|_{e_1} = 1$ and $\overline \phi|_{\overline{e_3}} = 0$. Then
    \[
      (\overline v \overline \phi)|_{e_1} = v|_{e_1}, \quad \overline v \overline \phi \in H^1_{0, \overline{e_2} \cup \overline{e_3}}(\overline T) \cap \Q_{q(p_a+1)+1}(\overline T),
    \]
    and we thus see that the function $\overline{\overline v}$ defined by
    $\overline{\overline v}|_{\hT_i} := v$, $\overline{\overline v}|_{\overline T} := \overline v \overline \phi$
    is a continuous polynomial extension of $v$ that moreover vanishes on $\overline{e_3}$. Its norm satisfies $\seminorm{\overline{\overline v}}_{\hT_i \cup \overline T} \lesssim \seminorm{v}_{\hT_i}$ (proof is analogous to~\cite[(3.29)]{Ern2015a}).
    The desired function $\underline v$ is found as the extension of $\overline{\overline v}$ across $e_4 \cup \overline{e_4}$.
  \item The proof of this case is analogous to that of (d).\qedhere
  \end{enumerate}
\end{proof}

\subsection*{Proof of Theorem~\ref{thm:reduction}} We proceed in several steps.
\subsubsection*{Step (A0)}
For every $\hT \in \hta$, we will find our functional $\dr_{\hT} \in H^1_*(\omega_a)'$ by constructing a Raviart-Thomas flux $\vec \sigma_{\hT} \in \RT_{p_a+1}(\hT)$, and write $\dr_{\hT}(v) = \ltwoinp{\vec \sigma_{\hT}}{\nabla v}{\hT}$. Let $\hT \in \hta$.

\subsubsection*{Step (A1)} Let us construct $\dr_{\hT}$. Lemma~\ref{lem:inverse} guarantees that there is a $\vec \sigma_{\hT}^{(1)} \in \RT_{p_a+1}(\hT)$ with
\begin{equation}
  \div \vec \sigma_{\hT}^{(1)} = \phi_{\hT}\quad \text{and} \quad \norm{\vec \sigma_{\hT}^{(1)}}_{\hT} \lesssim \norm{\ltwoinp{\phi_{\hT}}{\cdot}{\hT}}_{H^1_0(\hT)'}.
  \label{eqn:sigma1}
\end{equation}

By definition, $\dr_a$ has no contributions on the exterior edges of $\hta$. However, this $\vec \sigma_{\hT}^{(1)}$ \emph{can have} a nonzero normal component on edges in $\hE_{\hT}^{\ext}$. Let us resolve this inconsistency.

Without loss of generality, we may assume that $\hE_{\hT}^{\irr} \not= \emptyset$,\footnote{When $\hE_{\hT}^{\irr} = \emptyset$, then $\hta$ consists of a single element $\hT$, in which case $H^1_*(\omega_a) = H^1_{0, \hE^{\ext}_a}(\hT)$ so that we may invoke the assumption $\sup_p S^{(A)}_{\hE^{\ext}_a, p, q(p)} < \infty$ directly to find the saturation result~\eqref{eqn:reduction}.}
so the Galerkin problem
\[
  \honeinp{w_{\hT}}{v}{\hT} = F_{\hT}(v) := \sum_{\he \in \hE_{\hT}^{\ext}} \ltwoinp[s]{\vec \sigma_{\hT}^{(1)} \cdot \vec n_{\hT}}{v}{\he} \quad (v \in \sp H_{\hT}) \quad \text{where} \quad \boxed{\sp H_{\hT} := H^1_{0, \hE_{\hT}^{\irr}}(\hT)},
\]
has a unique solution $w_{\hT} \in \sp H_{\hT}$ for which it follows directly that
\[
  \div \nabla w_{\hT} = 0, \quad
  \seminorm{w_{\hT}}_{\hT} \leq \norm{F_{\hT}}_{\sp H_{\hT}'}, \quad
    \nabla w_{\hT} \cdot \vec n_{\hT} = - \vec \sigma_{\hT}^{(1)} \cdot \vec n_{\hT} \text{~~for~~} \he \in \hE_{\hT}^{\ext}.
\]
Now, integration by parts tells us that
\[
  F_{\hT}(v) = \ltwoinp{\div \vec \sigma^{(1)}_{\hT}}{v}{\hT} + \ltwoinp{\vec \sigma^{(1)}_{\hT}}{\nabla v}{\hT} \quad \left(v \in \sp H_{\hT}\right).
\]
Then, through~\eqref{eqn:sigma1} and (for the final inequality) $H^1_0(\hT) \subset \sp H_{\hT}$, we have
\[
  \norm{F_{\hT}}_{\sp H_{\hT}'} \leq \norm{\ltwoinp{\div \vec \sigma^{(1)}_{\hT}}{\cdot}{\hT}}_{\sp H_{\hT}'} + \norm{\vec \sigma_{\hT}^{(1)}}_{\hT} \lesssim \norm{\ltwoinp{\phi_{\hT}}{\cdot}{\hT}}_{\sp H_{\hT}'} + \norm{\ltwoinp{\phi_{\hT}}{\cdot}{\hT}}_{H^1_0(\hT)'} \leq 2 \norm{\ltwoinp{\phi_{\hT}}{\cdot}{\hT}}_{\sp H_{\hT}'},
\]
so that $\seminorm{w_{\hT}}_{\hT} \lesssim \norm{\ltwoinp{\phi_{\hT}}{\cdot}{\hT}}_{\sp H_{\hT}'}$. We then invoke Lemma~\ref{lem:rtext} with $\vec \tau := \nabla w_{\hT}$, and $\varphi \in L_2(\partial T)$ piecewise defined on $\he \in \hE_{\hT}^{\ext}$ as $\varphi|_{\he} := - \vec \sigma^{(1)}_{\hT} \cdot \vec n_{\hT}$ and on the rest of $\partial \hT$ such that $\varphi$ has mean zero, to find
a $\vec \sigma^{(2)}_{\hT} \in \RT_{p_a+1}(\hT)$ for which
\begin{equation}
  \div \vec \sigma_{\hT}^{(2)} = 0, \quad \vec \sigma_{\hT}^{(2)} \cdot \vec n_{\hT} = - \vec \sigma_{\hT}^{(1)} \cdot \vec n_{\hT} \text{~~for~~}\he \in \hE_{\hT}^{\ext}, \quad 
  \norm{\vec \sigma_{\hT}^{(2)}}_{\hT} \lesssim \norm{\ltwoinp{\phi_{\hT}}{\cdot}{\hT}}_{\sp H_{\hT}'},
  \label{eqn:localrtbound}
\end{equation}
so that $\vec \sigma_{\hT} := \vec \sigma_{\hT}^{(1)} + \vec \sigma_{\hT}^{(2)}$ has bounded norm, with normal components vanishing on $\hE^{\ext}_{\hT}$.

We then define $\dr_{\hT} \in H^1_*(\omega_a)'$ and $\dr^{(0)}_a \in H^1_*(\omega_a)'$ as
\[
  \dr_{\hT}(v) := \ltwoinp{\vec \sigma_{\hT}}{\nabla v}{\hT} \quad \text{and} \quad
  \dr^{(0)}_a := \sum_{\hT \in \hta} \dr_{\hT}.
\]

\subsubsection*{Step (A2)} We will verify~\eqref{eqn:r0props}. Integration by parts yields that for $v \in H^1_*(\omega_a)$,
\[
  \dr_{\hT}(v) = -\ltwoinp{\phi_{\hT}}{v}{\hT} + \sum_{\he \in \hE^{\irr}_{\hT}} \ltwoinp[s]{\phi^{(0)}_{\hT, \he}}{v}{\he},
\quad \text{where} \quad
  \phi^{(0)}_{\hT, \he} := \vec \sigma_{\hT} \cdot \vec n_{\hT} \in \P_{p_a+1}(\he).
\]
Therefore, $\dr^{(0)}_a$ removes all element contributions from $\dr_a$; it follows that indeed, $\dr_a - \dr^{(0)}_a$ is a sum of contributions over (interior) edges:
\[
  \dr_a(v) - \dr_a^{(0)}(v) = \sum_{\he \in \hE_a^{\irr}} \ltwoinp[s]{\phi_{\he}}{v}{\he} - \sum_{\hT \in \hta} \sum_{\he \in \hE_{\hT}^{\irr}} \ltwoinp[s]{\phi^{(0)}_{\hT,\he}}{v}{\he} =: \sum_{\he \in \hE_a^{\irr}} \ltwoinp[s]{\phi^{(0)}_{\he}}{v}{\he} \quad (v \in H^1_*(\omega_a)).
\]
Now, every $\phi^{(0)}_{\he}$ is a sum of polynomials $\phi^{(0)}_{\hT, \he}$, so indeed $\phi^{(0)}_{\he} \in \P_{p_a+1}(\he)$.

\subsubsection*{Step (A3)} We verify~\eqref{eqn:localnormbound}. By definition, $\dr_{\hT}(\bbone) = 0$.
Cauchy-Schwarz, \eqref{eqn:sigma1}, and~\eqref{eqn:localrtbound} imply
\[
  \norm{\dr_{\hT}}_{H^1_*(\omega_a)'} \leq \norm{\vec \sigma_{\hT}}_{\hT} = \norm{\vec \sigma_{\hT}^{(1)} + \vec \sigma_{\hT}^{(2)}}_{\hT} \lesssim \norm{\ltwoinp{\phi_{\hT}}{\cdot}{\hT}}_{\sp H_{\hT}'}.
\]
Moreover, $\hE_{\hT}^{\irr} \not= \emptyset$, hence it can be identified with a set $\E \in \mathbb E^{(A)}$ from~\eqref{eqn:Es}. By assumption, $\sup_p S^{(A)}_{\E, p+1, q(p+1)} \leq \hat S$, so that through Lemma~\ref{lem:refsat}, we find
\[
  \norm{\ltwoinp{\phi_{\hT}}{\cdot}{\hT}}_{\sp H_{\hT}'} \lesssim \norm{\ltwoinp{\phi_{\hT}}{\cdot}{\hT}}_{\sp V_{\hT}'}, \quad \text{where} \quad \boxed{\sp V_{\hT} := \sp H_{\hT} \cap \Q_{q(p_a + 1)}(\hT)}.
\]

Every $v \in \sp V_{\hT}$ vanishes on interior edges; write its zero-extension to $\omega_a$ as $\overline v \in H^1(\omega_a)$. Then
\[
  \hprime{q(p_a+1)} \ni
  \overline{\overline v} := \begin{cases}
    \overline v - \ltwoinp{\overline v}{\bbone}{\omega_a} & a \in \V_\T^{\irr}, \\
    \overline v & a \in \V_\T^{\ext}.
  \end{cases}
\]
By $\dr_a(\bbone) = 0$ for $a \in \V^{\irr}_\T$, we have $\ltwoinp{\phi_{\hT}}{v}{\hT} = \dr_a(\overline v) = \dr_a(\overline{\overline v})$; moreover, $\seminorm{v}_{\hT} = \seminorm{\overline{\overline v}}_{\omega_a}$, so
\[
  \norm{\ltwoinp{\phi_{\hT}}{\cdot}{\hT}}_{\sp V_{\hT}'} := \sup_{\mathclap{0 \not= v \in \sp V_{\hT}}} \frac{\abs{\ltwoinp{\phi_{\hT}}{v}{\hT}}}{\seminorm{v}_{\hT}} = \sup_{\mathclap{0 \not= v \in \sp V_{\hT}}} ~~ ~~ \frac{\abs{\dr_a(\overline{\overline v})}}{\seminorm{\overline{\overline v}}_{\omega_a}} \leq \sup_{0 \not= w \in \hprime{q(p_a+1)}} \frac{\abs{\dr_a(w)}}{\seminorm{w}_{\omega_a}} =: \norm{\dr_a}_{\hprime{ q(p_a+1)}'}.
\]
Chaining the dual norm inequalities in this step yields~\eqref{eqn:localnormbound}.

\subsubsection*{Step (B0)}
In this step, we traverse the interior edges $\hE_a^{\irr}$ in the order $(\he_j)_{j=1}^{n_a}$ constructed in \S\ref{sub:edgetrav}. Let $(\hT_j)_{j=1}^{n_a}$ be the sequence of squares for each $\he_j$.

At each iteration of the traversal, we use result (1) of Lemma~\ref{lem:stepb} to remove the $\he_i$-contribution from the previous residual by---in a fashion similar to \textit{Step (A1)}---solving a local Galerkin problem and constructing a Raviart-Thomas flux $\vec \sigma_i$ with specific properties. The resulting functional $\dr_{\he_i} \in H^1_*(\omega_a)'$ will be found as $\ltwoinp{\vec \sigma_i}{\nabla v}{\hT_i}$. We then use result (2) of the Lemma to establish the dual norm bound of~\eqref{eqn:rinormbound}, similar to \textit{Step (A3)}.

We will continue by induction. Let $i=1$.

\subsubsection*{Step (B1)}
We construct $\dr_{\he_i}$. By result (1) of Lemma~\ref{lem:stepb}, we have $\hE^{\loc,D}_{a,i} \not= \emptyset$, so the problem
\begin{equation}
  \honeinp{w^{(i)}}{v}{\hT_i} = \ltwoinp[s]{\phi^{(i-1)}_{\he_i}}{v}{\he_i} \quad (v \in \sp H_i) \quad \text{where} \quad \boxed{\sp H_i := H^1_{0, \E_{a,i}^{\loc,D}}(\hT_i)}
  \label{eqn:rilocalproblem}
\end{equation}
has a unique solution $w^{(i)} \in \sp H_i$ for which it holds that
\[
  \div \nabla w^{(i)} = 0, \quad
  \seminorm{w^{(i)}}_{\hT_i} \leq \norm{\ltwoinp{\phi^{(i-1)}_{\he_i}}{\cdot}{\he_i}}_{\sp H_i'}, \quad
  \begin{cases}
    \nabla w^{(i)} \cdot \vec n_{\hT_i} = - \phi^{(i-1)}_{\he_i} \text{ on $\he_i$}, \\
    \nabla w^{(i)} \cdot \vec n_{\hT_i} = 0 \text{ on $\set{ \he_j \in \hE^{\irr}_{\hT_i} : j < i}$}, \\
    \nabla w^{(i)} \cdot \vec n_{\hT_i} = 0 \text{ on $\hE^{\ext,N}_{\hT_i}$}.
  \end{cases}
\]
By Lemma~\ref{lem:rtext}, there is a $\vec \sigma_i \in \RT_{p_a+1}(\hT_i)$ with the same normal components as $\nabla w^{(i)}$, with
\begin{equation}
  \div \vec \sigma_i = 0, \quad \norm{\vec \sigma_i}_{\hT_i} \lesssim \norm{\ltwoinp{\phi^{(i-1)}_{\he_i}}{\cdot}{\he_i}}_{\sp H_i'}.
  \label{eqn:sigmaibound}
\end{equation}
We then define $\dr_{\he_i} \in H^1_*(\omega_a)'$ and $\dr^{(i)}_a \in H^1_*(\omega_a)'$ as
\[
  \dr_{\he_i}(v) := \ltwoinp{\vec \sigma_i}{\nabla v}{\hT_i} \quad \text{and} \quad \dr^{(i)}_a(v) := \dr^{(i-1)}_a(v) + \dr_{\he_i}(v) \quad (v \in H^1_*(\omega_a)).
\]

\subsubsection*{Step (B2)}
Let's look at~\eqref{eqn:riprop}. In light of~\eqref{eqn:r0props} when $i=1$, or $\eqref{eqn:riprop}$ for $i \geq 2$, suppose we have
\begin{equation}
  \dr_a(v) - \dr^{(i-1)}_a(v) = \sum_{j \geq i} \ltwoinp[s]{\phi^{(i-1)}_{\he_j}}{v}{\he_j} \text{~~for some~~} \phi^{(i-1)}_{\he_j} \in \P_{p_a+1}(\he_j) \quad (v \in H^1_*(\omega_a)).
  \label{eqn:previousrisum}
\end{equation}
Using that $v \in H^1_*(\omega_a)$ vanishes along edges in $\hE^{\ext,D}_{\hT_i}$, and considering the normal components of $\vec \sigma_i$, integration by parts yields~\eqref{eqn:riprop}:
\begin{align*}
  \dr_a(v) - \dr^{(i)}_a(v) &= \left[ \dr_a(v) - \dr^{(i-1)}_a(v)\right] - \dr_{\he_i}(v) = \sum_{j \geq i} \ltwoinp[s]{\phi^{(i-1)}_{\he_j}}{v}{\he_j} - \ltwoinp{\vec \sigma_i \cdot \vec n_{\hT_i}}{v}{\partial \hT_i} \\
  &= \sum_{j \geq i} \ltwoinp[s]{\phi^{(i-1)}_{\he_j}}{v}{\he_j} - \ltwoinp{\vec \sigma_i \cdot \vec n_{\hT_i}}{v}{\he_i} - \sum_{\mathclap{\he_j \in \hE_{\hT_i}^{\irr}: j > i}} \ltwoinp{\vec \sigma_i \cdot \vec n_{\hT_i}}{v}{\he_j} \\
  &= \sum_{\mathclap{j \geq i+1}} \ltwoinp[s]{\phi^{(i)}_{\he_j}}{v}{\he_j} - \sum_{\mathclap{\he_j \in \hE_{\hT_i}^{\irr}: j > i}} \ltwoinp[s]{\vec \sigma_i \cdot \vec n_{\hT_i}}{v}{\he_j} 
  =: \sum_{j \geq i+1} \ltwoinp[s]{\phi^{(i)}_{\he_j}}{v}{\he_j} \quad \text{for some} \quad \phi^{(i)}_{\he_j} \in \P_{p_a+1}(\he_j).
\end{align*}

\subsubsection*{Step (B3)}
We verify~\eqref{eqn:rinormbound}. By definition, $\dr_{\he_i}(\bbone) = 0$. Moreover, by result (1) of Lemma~\ref{lem:stepb}, $\hE^{\loc,D}_{a,i}$ corresponds with an $\E \in \mathbb E^{(B)}$ from~\eqref{eqn:Es}.\footnote{The set $\hE^{\loc,D}_{a,i}$ is in one of five states, whereas $\mathbb E^{(B)}$ has four; situation (a) and (b) of Lemma~\ref{lem:stepb} correspond with the same $\E \in \mathbb E^{(B)}$.} By assumption, $S^{(B)}_{\E, p_a+1, q(p_a+1)} \leq \hat S$, so~\eqref{eqn:sigmaibound} and Lemma~\ref{lem:refsat} yield
\[
  \norm{\dr_{\he_i}}_{H^1_*(\omega_a)'} \leq \norm{\vec \sigma_i}_{\hT_i} \lesssim \norm{\ltwoinp{\phi^{(i-1)}_{\he_i}}{\cdot}{\he_i}}_{\sp H_i'} \lesssim \norm{\ltwoinp{\phi^{(i-1)}_{\he_i}}{\cdot}{\he_i}}_{\sp V_i'}, \quad \text{where} \quad \boxed{\sp V_i := \sp H_i \cap \Q_{q(p_a + 1)}(\hT_i)}.
\]

To establish~\eqref{eqn:rinormbound}, it suffices to show
\begin{equation}
  \norm{\ltwoinp{\phi^{(i-1)}_{\he_i}}{\cdot}{\he_i}}_{\sp V_i'} \lesssim \norm{\dr_a - \dr^{(i-1)}_a}_{\hprime{q(p_a+1)+1}'},
  \label{eqn:Finormbound}
\end{equation}
because (by $\dr_a - \dr^{(i-1)}_a = \dr_a - \dr_a^{(0)} + \sum_{j=1}^{i-1} \dr_{\he_j}$, the triangle inequality, \eqref{eqn:rinormbound}, and $\# \hta \leq 16$),
\[
  \norm{\dr_a - \dr^{(i-1)}_a}_{\hprime{q(p_a+1)+1}'} \lesssim
  \norm{\dr_a - \dr^{(0)}_a}_{\hprime{q(p_a+1)+1}'} \lesssim
  \norm{\dr_a}_{\hprime{q(p_a+1)+1}'}.
\]

We proceed as in \textit{Step (A3)}.
Take $v \in \sp V_i$. Result (2) of Lemma~\ref{lem:stepb} guarantees a bounded extension from $v$ to a $Ev \in H^1(\omega_a) \cap \Q_{q(p_a+1)+1}^{-1}(\hta)$ that vanishes on interior edges $\he_j$ with $j > i$.
Moreover, $Ev$ is zero on edges $\he \in \hE^{\ext,D}_a$ whenever $a \in \V_\T^{\ext}$, so that in fact
\[
  \hprime{q(p_a+1)+1} \ni
  \overline{\overline v} := \begin{cases}
    Ev - \ltwoinp{Ev}{\bbone}{\omega_a} & a \in \V_\T^{\irr}, \\
    Ev & a \in \V_{\T}^{\ext},
  \end{cases}
\]
with $\seminorm{\overline{\overline v}}_{\omega_a} = \seminorm{Ev}_{\omega_a} \lesssim \seminorm{v}_{\hT}$.

Now, $\dr_a(\bbone) - \dr^{(i-1)}_a(\bbone) = 0$ when $a \in \V^{\irr}_\T$, so
$\dr_a(\overline{\overline v}) - \dr^{(i-1)}_a(\overline{\overline v}) = \dr_a(Ev) - \dr^{(i-1)}_a(Ev)$. Moreover, $Ev|_{\he_i} = v|_{\he_i}$ and $Ev|_{\he_j} = 0$ for $j > i$, so almost all terms of~\eqref{eqn:previousrisum} vanish when we plug in $Ev$, which yields $\dr_a(Ev) - \dr^{(i-1)}_a(Ev) = \ltwoinp{\phi^{(i-1)}_{\he_i}}{Ev}{\he_i} = \ltwoinp{\phi^{(i-1)}_{\he_i}}{v}{\he_i}$. Then, \eqref{eqn:Finormbound} follows by
\[
  \norm{\ltwoinp{\phi^{(i-1)}_{\he_i}}{\cdot}{\he_i}}_{\sp V_i'} = \sup_{\mathclap{0 \not= v \in \sp V_i}} \frac{\abs{\ltwoinp{\phi^{(i-1)}_{\he_i}}{v}{\he_i}}}{\seminorm{v}_{\hT_i}} \lesssim \sup_{\mathclap{0 \not= v \in \sp V_i}} \frac{\abs{\dr_a(\overline{\overline v}) - \dr^{(i-1)}_a(\overline{\overline v})}}{\seminorm{\overline{\overline v}}_{\omega_a}} \leq \norm{\dr_a - \dr^{(i-1)}_a}_{\hprime{q(p_a+1)+1}'}.
\]

\subsubsection*{Step (B4)} We repeat \textit{Steps (B1)--(B3)} for $i \in \set{2, \ldots, n_a-1}$, at each step finding functionals $\dr_{\he_i}$ and $\dr^{(i)}_a$ for which~\eqref{eqn:rinormbound} and~\eqref{eqn:riprop} hold.

\subsubsection*{Step (C)} When $a \in \V_{\T}^{\ext}$, the results of Lemma~\ref{lem:stepb} are satisfied once more for $i=n_a$. This allows us to repeat \textit{Steps (B1)--(B3)} for a $\dr_{\he_{n_a}} \in H^1_*(\omega_a)'$ satisfying~\eqref{eqn:rinormbound} and~\eqref{eqn:riprop}.

When $a \in \V_{\T}^{\irr}$, it is not possible to continue the iteration; the set $\E^{\loc,D}_{n_a}$ is empty so we cannot solve~\eqref{eqn:rilocalproblem}. However, we do know from~\eqref{eqn:riprop} that for $v \in H^1_*(\omega_a)$,
\[
  \dr_a(v) - \dr^{(n_a-1)}_a(v) = \ltwoinp[s]{\phi^{(n_a-1)}_{\he_{n_a}}}{v}{\he_{n_a}} \quad \text{for some} \quad \phi^{(n_a-1)}_{\he_{n_a}} \in \P_{p_a+1}(\he_j).
\]
Noting $\seminorm{v - \ltwoinp{v}{\bbone}{\hT_{n_a}}}_{\hT_{n_a}} \leq \seminorm{v}_{\omega_a}$ and $\ltwoinp{\phi^{(n_a-1)}_{\he_{n_a}}}{\bbone}{\he_{n_a}} = 0$ (by $\dr_a(\bbone) = 0 = \dr^{(n_a-1)}_a(\bbone)$), we find
\[
  \norm{\ltwoinp{\phi^{(n_a-1)}_{\he_{n_a}}}{\cdot}{\he_{n_a}}}_{H^1_*(\omega_a)'} \leq \norm{\ltwoinp{\phi^{(n_a-1)}_{\he_{n_a}}}{\cdot}{\he_{n_a}}}_{\sp H_{n_a}'} \quad \text{where} \quad \boxed{\sp H_{n_a} := H^1(\hT_{n_a})/\R}.
\]
The fact $\ltwoinp{\phi^{(n_a-1)}_{\he_{n_a}}}{\bbone}{\he_{n_a}} = 0$ also allows us to use $S^{(C)}_{p_a, q(p_a+1)} \leq \hat S$ and invoke Lemma~\ref{lem:refsat} yielding
\[
  \norm{\ltwoinp{\phi^{(n_a-1)}_{\he_{n_a}}}{\cdot}{\he_{n_a}}}_{\sp H_{n_a}'} \lesssim \norm{\ltwoinp{\phi^{(n_a-1)}_{\he_{n_a}}}{\cdot}{\he_{n_a}}}_{\sp V_{n_a}'} \quad \text{where} \quad \boxed{\sp V_{n_a} := \sp H_{n_a} \cap \Q_{q(p_a+1)}(\hT_{n_a})}.
\]

Any $v \in \sp V_{n_a}$ has zero mean, so reflecting across every row and column of $\hta$ yields a mean-zero extension $\overline v \in \hprime{ q(p_a+1)}$ with, by $\# \hta \leq 16$, norm $\seminorm{\overline v}_{\omega_a} = \sqrt{\# \hta}\seminorm{v}_{\hT_{n_a}} \lesssim \seminorm{v}_{\hT_{n_a}}$.

Finally, by $\ltwoinp{\phi^{(n_a-1)}_{\he_{n_a}}}{\overline v}{\he_{n_a}} = \ltwoinp{\phi^{(n_a-1)}_{\he_{n_a}}}{v}{\he_{n_a}}$, we have the dual norm bound
\[
  \norm{\ltwoinp{\phi^{(n_a-1)}_{\he_{n_a}}}{\cdot}{\he_{n_a}}}_{\sp V_{n_a}'}
  = \sup_{\mathclap{0\not= v \in \sp V_{n_a}}} \frac{\abs{\ltwoinp{\phi^{(n_a-1)}_{\he_{n_a}}}{\overline v}{\he_{n_a}}}}{\seminorm{v}_{\hT_{n_a}}}
  \lesssim \sup_{\mathclap{0\not= v \in \sp V_{n_a}}} \frac{\abs{\ltwoinp{\phi^{(n_a-1)}_{\he_{n_a}}}{\overline v}{\he_{n_a}}}}{\seminorm{\overline v}_{\omega_a}}
  \leq \norm{\ltwoinp{\phi^{(n_a-1)}_{\he_{n_a}}}{\cdot}{\he_{n_a}}}_{\hprime{q(p_a+1)}'};
\]
then, through the triangle inequality, we find
\[
  \norm{\ltwoinp{\phi^{(n_a-1)}_{\he_{n_a}}}{\cdot}{\he_{n_a}}}_{\hprime{ q(p_a + 1)}'} \leq \norm{\dr_a}_{\hprime{ q(p_a + 1)}'} + \norm{\dr^{(n_a-1)}_a}_{\hprime{ q(p_a + 1)}'} \lesssim \norm{\dr_a}_{\hprime{ q(p_a+1)}'}.
\]
Chaining the dual norm inequalities in \textit{Step (C)} then yields the desired bound~\eqref{eqn:drdiffbound}.  \hfill$\square$

\section{Computation of reference saturation coefficients}
In this section, we detail on the numerical computation of the saturation coefficient $S(\hat{\sp H}, \hat{\sp V}, \hat{ \sp F})$. To allow computation of this coefficient, we first write it as the solution to a generalized Eigenvalue problem. We then discuss the construction of bases for the spaces $\hat{\sp H}, \hat{\sp V},$ and $\hat{\sp F}$ involved in the specific saturation coefficients of Theorem~\ref{thm:reduction}.

\subsection{An equivalent problem}
In our applications, $\hat{\sp F}$ is a \emph{finite-dimensional subspace} of $\hat{\sp H}'$ rather than just a subset, allowing us to write $S(\hat{\sp H}, \hat{\sp V}, \hat{\sp F})$ as
$\sup_{\set{\hat F \in \hat{\sp F}: \norm{\hat F}_{\hat{\sp V}'} = 1}} \norm{\hat F}_{\hat{\sp H}'}$. Since $\hat{\sp F}$ is finite-dimensional, this supremum is attained, so we may equivalently solve
\begin{equation}
  \text{Find the largest $0 < \mu = S(\hat{\sp H}, \hat{\sp V}, \hat{\sp F})$ s.t., for some $0 \not= \hat F \in \hat{\sp F}$,} \quad \norm{\hat F}_{\hat{\sp H}'}^2 = \mu^2 \norm{\hat F}_{\hat{\sp V}'}^2.
  \label{eqn:norm-problem}
\end{equation}

\begin{prop}[Equivalent generalized Eigenvalue problem]
  Let $\Xi_{\hat{\sp H}}$, $\Xi_{\hat{\sp V}}$, and $\Sigma_{\hat{\sp F}}$ be bases for the three spaces. For $\hat U \in \{ \hat{\sp H}, \hat{\sp V}\}$, denote the stiffness matrices as $\mat A_{\hat U} := \honeinp{\Xi_{\hat U}}{\Xi_{\hat U}}{\refsq} = \left[ \honeinp{\xi_{i, \hat U}}{\xi_{j, \hat U}}{\refsq} \right]_{i,j=1}^{\# \Xi_{\hat U}}$, and load matrices as \mbox{$\mat L_{\hat U} := \Sigma_{\hat{\sp F}}(\Xi_{\hat U}) := \left[\sigma_i(\xi_{j, \hat U})\right]_{i,j}$}. Then the above problem is equivalent to finding the largest generalized Eigenvalue $\mu^2$ of the system
  \begin{equation}
    \mat R_{\hat{\sp H}} \vec F = \mu^2 \mat R_{\hat{\sp V}} \vec F, \quad \text{where} \quad \mat R_{\hat U} := \mat L_{\hat U} \mat A_{\hat U}^{-\top} \mat L_{\hat U}^{\top}.
  \label{eqn:genev-problem}
  \end{equation}
\end{prop}
\begin{proof}
Let, for $\hat U \in \{\hat{\sp H}, \hat{\sp V}\}$ and $\hat F \in \hat{\sp F}$, the function $u_{\hat U} = u_{\hat U}(\hat F)$ be the unique solution to
\begin{equation}
  \honeinp{u_{\hat U}}{v_{\hat U}}{\refsq} = \hat F(v_{\hat U}) \quad (v_{\hat U} \in \hat U).
  \label{eqn:genev-galerkin}
\end{equation}
  Recalling that we equip $\hat U$ with $\seminorm{\cdot}_{\refsq} := \norm{\nabla \cdot}_{\refsq}$, and thus $\hat U'$ with the corresponding dual norm, we have $\norm{\hat F}_{\hat U'} = \seminorm{u_{\hat U}}_{\refsq}$. Write $\hat F$ as $\vec F^\top \Sigma_{\hat{\sp F}}$ and $u_{\hat U}$ as $u_{\hat U} = \vec u_{\hat U}^\top \Xi_{\hat U}$; then $\vec u_{\hat U}^\top \mat A_{\hat U} = \honeinp{u_{\hat U}}{\Xi_{\hat U}}{\refsq} = \hat F(\Xi_{\hat U}) = \vec F^\top \mat L_{\hat U}$, or $\mat A_{\hat U}^\top \vec u_{\hat U} = \mat L_{\hat U}^\top \vec F$. Now, $\mat A_{\hat U}$ is invertible, so $\vec u_{\hat U} = \mat A_{\hat U}^{-\top} \mat L_{\hat U}^{\top} \vec F$.
  We see $\norm{\hat F}_{\hat U'}^2 = \seminorm{u_{\hat U}}_{\refsq}^2 = \vec u_{\hat U}^\top \mat A_{\hat U} \vec u_{\hat U} = \vec F^{\top} \mat L_{\hat U} \mat A_{\hat U}^{-\top} \mat L_{\hat U}^\top \vec F = \vec F^\top \mat R_{\hat U} \vec F$, reducing problem~\eqref{eqn:norm-problem} to
  \[
    \text{Find largest $\mu > 0$ s.t.~for some $\vec F \not= \vec 0$,} \quad
    \vec F^\top \mat R_{\hat{\sp H}} \vec F = \mu^2 \vec F^\top \mat R_{\hat{\sp V}} \vec F
    \iff
    \mu^2 = \frac{\vec F^\top \mat R_{\hat{\sp H}} \vec F}{\vec F^\top \mat R_{\hat{\sp V}} \vec F},
  \]
  which, by virtue of both $\mat R_{\hat{\sp H}}$ and $\mat R_{\hat{\sp V}}$ being symmetric positive-definite, is a Rayleigh quotient for the generalized Eigenvalue problem of~\eqref{eqn:genev-problem} (cf.~\cite{Li2015}).
\end{proof}

\subsection{Discrete saturation coefficients}
In all of the cases of Theorem~\ref{thm:reduction}, the space $\hat{\sp H}$ is an infinite-dimensional closed subspace of $H^1(\refsq)$, so computing $S(\hat{\sp H}, \hat{\sp V}, \hat{ \sp F})$ by means of~\eqref{eqn:genev-problem} will likely not be possible. However, the following result shows that we may restrict ourselves to a finite-dimensional subspace that is large enough.

\begin{lemma}
  Since $\hat{\sp F}$ is a finite-dimensional subspace of $\hat{\sp H}'$, a compactness argument shows that the \emph{discrete saturation coefficient} $S(\hat{\sp H} \cap \Q_r(\refsq), \hat{\sp V}, \hat{\sp F})$ tends to $S(\hat{\sp H}, \hat{\sp V}, \hat{\sp F})$ for $r \to \infty$.
  \label{lem:discsat}
\end{lemma}

\subsection{Bases for the subspaces}
Solving~\eqref{eqn:genev-problem} hinges on computing the stiffness matrix $\mat A_{\hat {\sp H}}$ and load matrix $\mat L_{\hat {\sp H}}$, which depend on the choice of basis.
In practice, we are able to choose these bases with tensor-product structure. For instance, when $\Xi_{\hat{\sp H}} =: \Xi^1 \otimes \Xi^2$, we see
\[
  \mat A_{\hat{\sp H}} = \honeinpx{\Xi^1}{\Xi^1}{\hat I} \otimes \ltwoinp[x]{\Xi^2}{\Xi^2}{\hat I}
  + \ltwoinp[x]{\Xi^1}{\Xi^1}{\hat I} \otimes \honeinpx{\Xi^2}{\Xi^2}{\hat I},
\]
where $\hat I := [-1,1]$, so that $\refsq = \hat I \times \hat I$, and
with $\otimes$ denoting the Kronecker product. Essentially, computation of the saturation coefficient boils down to computing a number of inner products.

Define $L_k$ as the $k$th Legendre polynomial, with $\deg L_k = k$ and $L_k(1) = 1$. The functions
\[
  \varphi_k(x) := \sqrt{k + \tfrac{1}{2}} L_k(x), \quad (k \geq 0)
\]
then constitute an $L_2(\hat I)$-orthonormal basis called the \emph{Legendre} basis.
Moreover, the functions
\[
  \xi_k(x) := \sqrt{k - \tfrac{1}{2}} \int_x^1 L_{k-1}(s) \dif s = \frac{1}{\sqrt{4k-2}}(L_{k-2}(x) - L_k(x)), \quad (k \geq 2)
\]
constitute an orthonormal basis with respect to the $H^1(\hat I)$-seminorm which we call the \emph{Babu\v{s}ka-Shen} basis. With respect to the $L_2(\hat I)$-inner product, this basis is quasi-orthogonal in that
\[
  \ltwoinp[x]{\xi_k}{\xi_m}{\hat I} = 0 \iff k-m \not\in \set{-2, 0, 2}, \quad \text{and} \quad
  \ltwoinp[x]{\varphi_k}{\xi_m}{\hat I} = 0 \iff k-m \not\in \set{0, 2}.
\]
We can supplement the Babu\v{s}ka-Shen basis with $\xi_1(x) = \frac{1}{2}\sqrt{2} (1-x)$ to find an orthonormal basis for $H^1_{0, \set{1}}(\hat I)$, and with $\tilde \xi_1(x) := \xi_1(-x)$ for a basis for $H^1_{0, \set{-1}}(\hat I)$. These supplemented functions are $L_2(\hat I)$-orthogonal to $\xi_m$ for $m \geq 4$.

Recall the saturation coefficients from Theorem~\ref{thm:reduction}. The space $\hat{\sp V}$ is in every case just $\hat{\sp H}$ restricted to polynomials of lower degree, so we will focus on building bases for $\hat{\sp H}$ and $\hat{\sp F}$.

\subsubsection{First discrete coefficient $S^{(A)}_{\E, p, q, r}$} Denote 
\[
  \hat{\sp H} := H^1_{0, \E}(\refsq) \cap \Q_r(\refsq), \quad \hat{\sp F} := \{h \mapsto \ltwoinp{\phi}{h}{\refsq}: \phi \in \Q_p(\refsq)\} \subset \hat{\sp H}'.
\]
A tensorized basis $\Xi_{\hat{\sp H}} = \Xi^1 \otimes \Xi^2$ for $\hat{\sp H}$ is readily constructed through the Babu\v{s}ka-Shen basis, supplemented to account for boundary conditions, up to degree $r$ in each coordinate.

    Choosing $\Phi := \Phi^1 \otimes \Phi^2$ with $\Phi^1 := \Phi^2$ the Legendre basis up to degree $p$, we set $\Sigma_{\hat{\sp F}} := \ltwoinp{\Phi}{\cdot}{\refsq}$. Then, the load matrix can be computed from
    $\mat L_{\hat{\sp H}} = \ltwoinp{\Phi}{\Xi_{\hat{\sp H}}}{\refsq} = \ltwoinp[x]{\Phi^1}{\Xi^1}{\hat I} \otimes \ltwoinp[x]{\Phi^2}{\Xi^2}{\hat I}$.

\subsubsection{Second discrete coefficient $S^{(B)}_{\E, p, q, r}$} The space $\hat{\sp H}$ is the same as in $S^{(A)}_{\E, p, q, r}$ so its basis $\Xi_{\hat{\sp H}} = \Xi^1 \otimes \Xi^2$ is readily constructed.  For $\hat{\sp F} := \set{h \mapsto \ltwoinp[s]{\phi}{h}{\hat e_1} : \phi \in \P_p(\hat e_1)}$, we choose the basis $\Sigma_{\hat{\sp F}} := \ltwoinp[s]{\Phi}{\cdot}{\hat e_1}$ with $\Phi$ the Legendre basis for $\P_p(\hat e_1)$. Then
    \begin{equation}
      \mat L_{\hat{\sp H}} =
      \ltwoinp[s]{\Phi}{\Xi_{\hat {\sp H}}}{\hat e_1} = \Xi^1(1) \otimes \ltwoinp[x]{\Phi}{\Xi^2}{\hat I} \quad \text{where} \quad \Xi^1(x) := \left(\xi(x)\right)_{\xi \in \Xi^1}.
      \label{eqn:F2basis}
    \end{equation}
    The polynomials $\xi \in \Xi^1$ with $\deg \xi \geq 2$ have $\xi(1) = 0$, so $\mat L_{\sp H}$ is sparse with entire zero rows.
    
\subsubsection{Third discrete coefficient $S^{(C)}_{p,q,r}$} To create a basis for $\hat{\sp H} := \Q_r(\refsq)/\R$, we first construct
\[
  X := \set{\chi_0, \xi_1 - \ltwoinp{\xi_1}{\bbone}{\hat I}, \xi_2 - \ltwoinp{\xi_2}{\bbone}{\hat I}, \ldots, \xi_r - \ltwoinp{\xi_r}{\bbone}{\hat I}}, \quad \text{where} \quad \chi_0 := \bbone/\sqrt{2}
\]
which is a basis for $\Q_r(\hat I)$, (almost) orthogonal w.r.t.~the $H^1(\hat I)$-seminorm, with every element except $\chi_0$ having zero mean.
The set $\Xi_{\hat{\sp H}} := X \times X \setminus \set{\chi_0 \otimes \chi_0}$ then consists of linearly independent polynomials with zero mean, and is of correct cardinality, hence a basis for $\hat{\sp H}$.
    
Legendre polynomials $\phi_k$ of degree $k \geq 1$ have mean zero, so \mbox{$\Phi_* := \set{\phi_k: 1 \leq k \leq p}$} is a basis for $\P_p(\hat e_1)/\R$, and $\Sigma_{\hat{\sp F}} := \ltwoinp{\Phi_*}{\cdot}{\hat e_1}$ a basis for $\hat{\sp F} := \set{h \mapsto \ltwoinp{\phi}{h}{\hat e_1} : \phi \in \P_p(\hat e_1)/\R}$. Its load matrix is formed analogously to~\eqref{eqn:F2basis}.

\section{Numerical results}
\label{sec:numres}
In Theorem~\ref{thm:reduction}, we showed that patch-based $p$-robust saturation holds, under the assumption that a number of quantities on the reference square are finite. More specifically, we are interested in finding a function $q: \N \to \N$ such that the \emph{saturation coefficients} 
\[
  S^{(A)}_{\E, p, q(p)}, \quad S^{(B)}_{\mathcal F, p, q(p)}, \quad S^{(C)}_{p, q(p)} \quad (\E \in \mathbb E^{(A)}, ~~ \mathcal F \in \mathbb E^{(B)})
\]
(cf.~Thm.~\ref{thm:reduction}) are uniformly bounded in $p$. Unable to compute the limit $p \to \infty$, we resort to computing them for a number of large but finite values of $p$, and extrapolate from this progression. Moreover, with our current approach, we are unable to compute the above quantities, so we instead compute the \emph{discrete saturation coefficients}
\[
  S^{(A)}_{\E, p, q, r}, \quad S^{(B)}_{\mathcal F, p, q, r}, \quad S^{(C)}_{p, q, r} \quad (\E \in \mathbb E^{(A)}, ~~ \mathcal F \in \mathbb E^{(B)})
\]
(cf.~Lemma~\ref{lem:discsat}) for some values of $r$ that are large relative to $p$ and $q$, and expect to see \emph{$r$-stabilization} of the discrete coefficient to its ``continuous'' counterpart.

In~\cite{Canuto2017}, it was shown that (for a slightly different setting), a strategy of the form $q(p) = p+n$ with $n \in \N$ is insufficient, whereas for \emph{any} $\lambda > 0$, the choice $q(p) = p + \lceil \lambda p \rceil$ exhibits saturation. This motivates our choice to run numerical computations for
\[
  q(p) = p+4, \quad q(p) = p + \lceil p/7 \rceil, \quad q(p) = 2p.
\]

By symmetry, there are five fundamentally different configurations of sets of edges in $\mathbb E^{(A)}$:
\[
  \E_1 := \set{\hat e_1}, \quad \E_2 := \set{\hat e_1, \hat e_2}, \quad \E_3 := \set{\hat e_1, \hat e_3}, \quad \E_4 := \set{\hat e_1, \hat e_2, \hat e_3}, \quad \E_5 := \set{\hat e_1, \hat e_2, \hat e_3, \hat e_4}.
\]
Moreover, enumerating the elements of $\mathbb E^{(B)}$ as
\[
  \mathcal F_1 := \set{\hat e_2}, \quad \mathcal F_2 := \set{\hat e_3}, \quad \mathcal F_3 := \set{\hat e_2, \hat e_3}, \quad \mathcal F_4 := \set{\hat e_2, \hat e_3, \hat e_4},
\]
we conclude that there are $5+4+1=10$ reference problems to investigate.

Results were gathered using the sparse matrix library \texttt{scipy.sparse} with \texttt{float64} matrices, using \texttt{scipy.sparse.linalg.spsolve} and \texttt{scipy.sparse.linalg.eigsh}, with default settings. Sparsity of the matrices ensures highly accurate results.

See Table~\ref{tbl:coeffs} for the computed results. First we study the $r$-stabilization by means of the three `hardest' problems (ordered by saturation coefficient for $p=4$, $q=8$, $r=16$). There is little difference between $r=2q$, $r=4q$, and $r=8q$, indicating that $r=2q$ is sufficient.

Choosing $q=p+4$ is insufficient for $p$-robust saturation: for every problem, the discrete saturation coefficients increase as a function of $p$. For the two strategies $q=p + \lceil p/7 \rceil$ and $q = 2p$, we see that these coefficients \emph{decrease} as a function of $p$, stringly suggesting $p$-robust saturation for $p \to \infty$. For $q = 2p$, these values even tend to $1$, indicating full saturation.

\begin{table}[t]
  \footnotesize
  \definecolor{Gray}{gray}{0.8}
\newcolumntype{g}{>{\columncolor{Gray}}r}
\newcolumntype{H}{>{\setbox0=\hbox\bgroup}c<{\egroup}@{}}
\begin{tabular}{lr ggggH rrrH ggggH}
  && \multicolumn{4}{c}{$q(p) = p + 4$}             & \multicolumn{4}{c}{$q(p) = p + \lceil p/7 \rceil$}  && \multicolumn{4}{c}{$q(p) = 2p$}  \\
  && \multicolumn{4}{c}{$\overbrace{\hspace{4cm}}$} & \multicolumn{4}{c}{$\overbrace{\hspace{3cm}}$}      && \multicolumn{4}{c}{$\overbrace{\hspace{4cm}}$}  \\
\toprule
  & &      $p=4$   &       12  &       28  &       60  &      124 & $p=14$  &       28  &       56  &      112 & $p=4$   &      8   &      16  &      32  &      64  \\
  & $r$ &      $q=8$   &       16  &       32  &       64  &      128 & $q=16$  &       32  &       64  &      128 & $q=8$   &      16  &      32  &      64  &      128 \\
\midrule
  {\small $S^{(A)}_{\E_1, p, q, r}$}  & $2q$          &  1.0017 &  1.0344 &  1.1562 &  1.4015 &  --- & 1.1905 &  1.1562 &  1.1431 &  --- &  1.0017 &  1.0005 &  1.0003 &  1.0002 &  1.0002 \\\hline
  {\small $S^{(A)}_{\E_2, p, q, r}$  }& $2q$          &  1.0120 &  1.1076 &  1.3505 &  1.7715 &  --- & 1.3970 &  1.3505 &  1.3334 &  --- &  1.0120 &  1.0060 &  1.0042 &  1.0035 &  1.0032 \\\hline
  {\small $S^{(A)}_{\E_3, p, q, r}$  }& $2q$          &  1.0017 &  1.0350 &  1.1580 &  1.4039 &  --- & 1.1945 &  1.1580 &  1.1440 &  --- &  1.0017 &  1.0006 &  1.0003 &  1.0002 &  1.0002 \\\hline
  {\small $S^{(A)}_{\E_4, p, q, r}$ } & $2q$          &  1.0138 &  1.1112 &  1.3541 &  1.7744 &  --- & 1.4038 &  1.3541 &  1.3352 &  --- &  1.0138 &  1.0065 &  1.0044 &  1.0036 &  1.0033 \\\hline
  {\small $S^{(A)}_{\E_5, p, q, r}$}  & $2q$          &  1.0150 &  1.1143 &  1.3575 &  1.7772 &  --- & 1.4101 &  1.3575 &  1.3370 &  --- &  1.0150 &  1.0070 &  1.0046 &  1.0037 &  1.0033 \\\hline
   \multirow{3}{*}{{\small $S^{(B)}_{\mathcal F_1, p, q, r}$}}
                            & $2q$            &  1.0295 &  1.1012 &  1.2092 &  1.3429 &  --- & 1.2380 &  1.2092 &  1.1952 &  --- &  1.0295 &  1.0204 &  1.0165 &  1.0147 &  1.0138 \\
                            & $4q$            &  1.0317 &  1.1075 &  1.2196 &  1.3570 &  --- & 1.2502 &  1.2196 &  1.2048 &  --- &  1.0317 &  1.0218 &  1.0176 &  1.0156 &     --- \\
                            & $8q$            &  1.0318 &  1.1079 &  1.2203 &     --- &  --- & 1.2511 &  1.2203 &     --- &  --- &  1.0318 &  1.0219 &  1.0176 &     --- &     --- \\\hline
  {\small $S^{(B)}_{\mathcal F_2, p, q, r}$}  
                            & $2q$            &  1.0013 &  1.0106 &  1.0314 &  1.0634 &  --- & 1.0385 &  1.0314 &  1.0277 &  --- &  1.0013 &  1.0006 &  1.0004 &  1.0003 &  1.0003 \\\hline
  \multirow{3}{*}{{\small $S^{(B)}_{\mathcal F_3, p, q, r}$}}
                             & $2q$           &  1.0295 &  1.1012 &  1.2092 &  1.3429 &  --- & 1.2380 &  1.2092 &  1.1952 &  --- &  1.0295 &  1.0204 &  1.0165 &  1.0147 &  1.0138 \\
                             & $4q$           &  1.0317 &  1.1075 &  1.2196 &  1.3570 &  --- & 1.2502 &  1.2196 &  1.2048 &  --- &  1.0317 &  1.0218 &  1.0176 &  1.0156 &     --- \\
                             & $8q$           &  1.0318 &  1.1079 &  1.2203 &     --- &  --- & 1.2511 &  1.2203 &     --- &  --- &  1.0318 &  1.0219 &  1.0176 &     --- &     --- \\
                            \hline
   \multirow{3}{*}{{\small $S^{(B)}_{\mathcal F_4, p, q, r}$}} 
                             & $2q$           &  1.0346 &  1.1055 &  1.2118 &  1.3443 &  --- & 1.2430 &  1.2118 &  1.1965 &  --- &  1.0346 &  1.0226 &  1.0175 &  1.0152 &  1.0140 \\
                             & $4q$           &  1.0374 &  1.1123 &  1.2227 &  1.3587 &  --- & 1.2563 &  1.2227 &  1.2064 &  --- &  1.0374 &  1.0242 &  1.0186 &  1.0161 &     --- \\
                             & $8q$           &  1.0376 &  1.1128 &  1.2234 &     --- &  --- & 1.2572 &  1.2234 &     --- &  --- &  1.0376 &  1.0243 &  1.0187 &     --- &     --- \\
                            \hline
  {\small $S^{(C)}_{p, q, r}$}    & $2q$          &  1.0013 &  1.0106 &  1.0313 &  1.0634 &  --- & 1.0384 &  1.0313 &  1.0277 &  --- &  1.0013 &  1.0006 &  1.0004 &  1.0003 &  1.0003 \\
\bottomrule
\end{tabular}

  \vspace{0.5em}
  \caption{Discrete saturation coefficients for different $p, q$, and $r$. We discern three `bands' of columns, one for each function $q$, and within each band, different values of $p$, one per column. We moreover see a number of different `bands' of rows, one for each reference problem; within each band, a number of different discrete saturation coefficients are shown, one for each $(p, q, r)$-tuple.}
  \label{tbl:coeffs}
\end{table}

\section*{Acknowledgement}
The author wishes to thank their advisor Rob Stevenson for the many fruitful discussions.

\bibliographystyle{hsiam}

\end{document}